\documentclass[11pt,a4paper,oldfontcommands]{article}
\usepackage[utf8]{inputenc}
\usepackage[T1]{fontenc}
\usepackage{microtype}
\usepackage[dvips]{graphicx}
\usepackage{xcolor}
\usepackage{times}

\usepackage[english]{babel}
\usepackage{amsmath,mathtools}
\usepackage{amsfonts}
\usepackage{mathrsfs}
\usepackage{amsthm}
\usepackage{subfigure}
\usepackage{envmath}
\usepackage{amssymb}
\usepackage{dsfont}
\usepackage{bigints}
\usepackage{textgreek}
\usepackage{authblk}
\usepackage{titlesec}
\usepackage{enumitem}
\usepackage{url}

\usepackage[
breaklinks=true,
colorlinks=false,
bookmarks=true,bookmarksopenlevel=2]{hyperref}

\usepackage{geometry}
\newtheorem{theorem}{Theorem}[section]
\newtheorem{prop}[theorem]{Proposition}
\newtheorem{lem}[theorem]{Lemma}
\newtheorem{cor}[theorem]{Corollary}

\theoremstyle{remark}

\numberwithin{equation}{section}

\newcommand{\divg}{\operatorname{div}}

\newcommand*\diff{\mathop{}\!\mathrm{d}}

\DeclarePairedDelimiter\abs{\lvert}{\rvert}%
\DeclarePairedDelimiter\norm{\lVert}{\rVert}%

\makeatletter
\let\oldabs\abs
\def\abs{\@ifstar{\oldabs}{\oldabs*}}
\let\oldnorm\norm
\def\norm{\@ifstar{\oldnorm}{\oldnorm*}}
\makeatother

\newcommand*{\myemail}[1]{%
    \normalsize\href{mailto:#1}{#1}\par
    }

\allowdisplaybreaks

\titleformat{\section}[block]{\centering \scshape \large}{\thesection.}{0.3\baselineskip}{}
\titlespacing{\section}{0pt}{*5}{*2}

\titleformat{\subsection}[block]{\bfseries}{\thesubsection.}{.5em}{}
\titlespacing{\subsection}{0pt}{*2.5}{*1}

\titleformat{\subsubsection}[runin]{\itshape}{\normalfont \thesubsubsection.}{.5em}{}[.]
\titlespacing{\subsubsection}{0pt}{*2.5}{0.5em}

\titleformat{\section}[block]{\centering \scshape \large}{\thesection.}{0.3\baselineskip}{}
\titlespacing{\section}{0pt}{*5}{*2}

\titleformat{\subsection}[block]{\bfseries}{\thesubsection.}{.5em}{}
\titlespacing{\subsection}{0pt}{*2.5}{*1}

\titleformat{\subsubsection}[runin]{\itshape}{\normalfont \thesubsubsection.}{.5em}{}[.]
\titlespacing{\subsubsection}{0pt}{*2.5}{0.5em}

\title{Unbounded-energy solutions to the fluid+disk system and long-time behavior for large initial data.}
\date{\vspace{-1cm}}

\author[]{Guillaume Ferriere}
\author[]{Matthieu Hillairet}

\affil[]{Institut de Recherche Mathématique Avancée, UMR 7501 Université de Strasbourg et CNRS, France, \myemail{guillaume.ferriere@math.unistra.fr}
IMAG,  Univ Montpellier, CNRS, Montpellier, France, \myemail{matthieu.hillairet@umontpellier.fr}}

\begin{document}

\maketitle

\begin{abstract}
In this paper, we analyse the long-time behavior of solutions to a coupled system describing the motion of a rigid disk in a 2D viscous incompressible fluid. Following previous approaches in \cite{Ervedoza_Hillairet__disk_fluid,tucsnaktakahashi2004,WangXin2011} we look at the problem in the system of coordinates associated with the center of mass of the disk. 
Doing so, we introduce a further nonlinearity to the classical Navier Stokes equations. In comparison with the classical nonlinearities, this new term lacks time and space integrability, thus complicating strongly the analysis of the long-time behavior of solutions. 

We provide herein two refined tools : a refined analysis of the Gagliardo-Nirenberg inequalities and a thorough description of fractional powers of the so-called fluid-structure operator \cite{DMS}. On the basis of these two tools we extend decay estimates obtained in \cite{Ervedoza_Hillairet__disk_fluid} to arbitrary initial data and show local stability of the Lamb-Oseen vortex  in the spirit of \cite{Gallay_Maekawa__NS_obstacle,GallayWayne}.
\end{abstract}

\section{Introduction}

In this paper, we pursue the studies on the long-time behavior of solutions to the following model for the motion of a rigid disk inside a viscous incompressible fluid:
\begin{align}
    \frac{\partial u}{\partial t} + (u \cdot \nabla) u - \nu \Delta u + \nabla p = 0& \quad \textnormal{for } x \in \mathcal{F} (t),  \label{eq:momentum}\\
    \divg u = 0& \quad \textnormal{for } x \in \mathcal{F} (t), \label{eq:incompressibility}\\
    u (t, x) = h' (t) + \omega (t) (x - h(t))^\perp& \quad \textnormal{for } x \in \partial B (t), \label{eq:noslip} \\
    m h''(t) = - \int_{\partial B (t)} \Sigma (u,p) n \diff \sigma (x)& 
    \label{eq:linear}\\
    \mathcal{J} \omega' (t) = - \int_{\partial B (t)} (x - h(t))^\perp \cdot \Sigma (u,p) n \diff \sigma (x).&   \label{eq:angular}
 \end{align}
Here $u \in \mathbb R^2$ and $p\in \mathbb R$ stand for the velocity-field/pressure 
unknowns describing the behavior of a homogeneous incompressible viscous fluid. The rigid solid disk occupies the domain $B(t) := B(h(t),1)$ and its motion is described by a translation velocity ${\ell(t)} = h'(t)$ and a rotation velocity $\omega.$ Doing so, we prescribe the evolution of the fluid+disk system by integrating the incompressible Navier-Stokes equations \eqref{eq:momentum}-\eqref{eq:incompressibility} in the fluid domain $\mathcal F(t) := \mathbb R^2 \setminus \overline{B(t)}$ and the Newton
equation of solid dynamics \eqref{eq:linear}-\eqref{eq:angular}.
 We emphasize that the motion of the fluid and the solid are both unknowns. The system is complemented with no-slip interface conditions 
\eqref{eq:noslip} and transmission of normal stress. The stress tensor $\Sigma(u,p)$ appearing then
in the Newton laws is  the fluid stress tensor
\begin{equation*}
    \Sigma (u,p) = - p \operatorname{Id} + 2 \nu D(u),
\end{equation*}
with
\begin{equation*}
    D(u)_{i, j} = \frac{1}{2} \Bigl( \frac{\partial u_i}{\partial x_j} + \frac{\partial u_j}{\partial x_i} \Bigr), \qquad 1 \leq i, j \leq 2.
\end{equation*}
We remind that $\nu >0$ stands for the fluid viscosity and that, due to the incompressibility condition, the viscous operator appearing in \eqref{eq:momentum} reads:
$$
\nu \Delta u -\nabla p = {\rm div} \Sigma(u,p) 
$$
where, by convention, the divergence operator of a matrix is computed row-wise. 
By scaling arguments, we prescribed that the density of the fluid is constant equal to $1$ and that
the solid has radius $1.$ Below, it appears also that the viscosity $\nu$ has only an influence through a time-scaling so we fix $\nu=1$ for simplicity.
The quantity $m$ and $\mathcal J$ appearing in the Newon laws represent respectively the mass and inertia of the solid disk. In the 2D case under consideration here, the inertia $\mathcal J$ is time-independent. The symbol $n$ appearing in the integrals of 
\eqref{eq:linear}-\eqref{eq:angular} stands for the normal to $\partial B(t)$ inward $B(t).$
We keep the convention that the normal is directed outward the fluid domain throughout the paper.  { Like in \cite{Ervedoza_Hillairet__disk_fluid}, our} motivation for studying this system is
to analyse the energy exchange between the solid body and the rigid disk, we do not include any forcing term such as gravity in the system.  {We point out that, by a standard scaling argument, understanding the long-time behavior of solutions is also related to the small-body limit \cite{HeIftimie}.  Our analysis might then be adapted  to this second problem  as in \cite{LacaveTakahashi}.}
\medskip

Systems like \eqref{eq:momentum}--\eqref{eq:angular} coupling ODEs and PDEs and describing the motion of solid bodies inside a viscous fluid have been the subject of numerous studies in  the past years. Regarding the specific case of 
one rigid disk in an unbounded viscous fluid, the Cauchy theory for finite-energy initial data is studied in \cite{tucsnaktakahashi2004}. The authors remark therein that solutions to \eqref{eq:momentum}--\eqref{eq:angular} satisfy the {\em a priori} estimate:
$$
\dfrac{1}{2} \left[ m |\ell(t)|^2 + \mathcal{J}\omega(t)^2 + \int_{\mathcal F(t)} |u(t,\cdot)|^2 \right]
+\int_0^t \int_{\mathcal F(s)} |D(u)|^2 =\dfrac{1}{2} \left[ m |\ell(0)|^2 + \mathcal{J}\omega(0)^2 + \int_{\mathcal F(0)} |u(0,\cdot)|^2 \right] 
$$
This opens the way to the construction of  global-in-time finite-energy solution for arbitrary data. To this purpose, the authors operate the change of unknowns:
\begin{equation}
    v(t, x) = u (t, x - h(t)), \qquad \tilde{p} = p (t, x - h(t)), \qquad \ell (t) = h'(t), \qquad \omega (t) = \omega (t).
\end{equation}
and obtain the new system:
\begin{align}
    \frac{\partial v}{\partial t} + ((v - \ell(t)) \cdot \nabla) v - \Delta v + \nabla \tilde{p} = 0& \quad \textnormal{for } x \in \mathcal{F}_0, \label{eq:NS_mod1} \\
    \divg v = 0& \quad \textnormal{for }  x \in \mathcal{F}_0, \\
    v (t, x) = \ell (t) + \omega (t) \, x^\perp& \quad \textnormal{for }  x \in \partial B_0, \\
    m \ell'(t) = - \int_{\partial B_0} \Sigma (v,\tilde{p}) n \diff \sigma (x)& \\
    \mathcal{J} \omega' (t) = - \int_{\partial B (t)} x^\perp \cdot \Sigma (v,\tilde{p}) n \diff \sigma (x),&  \label{eq:NS_mod6}
 \end{align}
where $B_0  = B(0,1)$ and $\mathcal F_0 = \mathbb R^2 \setminus \overline{B}_0.$ 
With this change of unknowns, we have now a problem in a fixed geometry
that we can complete prescribing an initial condition. Setting an inital time $t_0 \geq 0$ that can be strictly positive, this condition reads:
\begin{equation}
   v|_{t=t_0} = v_0 \quad \textnormal{for } x \in \mathcal{F}_0, \quad
    \ell (t_0) = \ell_0, \quad \omega (t_0) = \omega_0. \label{eq:NS_mod7}
\end{equation}
 Despite \eqref{eq:NS_mod1}-\eqref{eq:NS_mod6} is an autonomous problem, we introduce here a generalized Cauchy-problem with arbitrary initial time. This  will  have an influence below because of our choice for initial data. We recall here also that the pressure $\tilde{p}$ can be seen as the Lagrange-multiplier of the divergence-free condition involved in the system above. For this reason, there is no initial condition on $\tilde{p}.$ In our formalism, the pressure will also be a secondary unknown that is taken rid {\em via} a projector argument and that can be recovered {\em a posteriori.} For all these reasons, we state  our results in terms of $(v,\ell,\omega)$ only.
For instance, in \cite{tucsnaktakahashi2004}, the authors consider the case $t_0=0.$ They consider initial data $(v_0,\ell_0,\omega_0) \in L^2(\mathcal F_0) \times \mathbb R^2 \times \mathbb R$ such that :
\begin{equation} \label{eq_compatinitdata}
{\rm div}v_0 = 0 \text{ in $\mathcal F_0$} \qquad v_0 \cdot n = (\ell_0 + \omega_0 x^{\bot}) \cdot n \text{ on $\partial B_0$}
\end{equation}
and
construct global-in-time finite-energy solutions in the sense that:
\begin{itemize}
\item $v \in C([0,\infty);L^2(\mathcal F_0)-w)$ with  $\nabla v \in L^2((0,\infty); L^2(\mathcal F_0))$
\item $(\ell,\omega) \in C([0,\infty);\mathbb R^3)$
\item $(v,\ell,\omega)$ solve \eqref{eq:NS_mod1}-\eqref{eq:NS_mod7} with the associated {\em a priori} estimate.
\end{itemize}
$$
\dfrac{1}{2} \left[ m |\ell(t)|^2 + \mathcal{J}\omega(t)^2 + \int_{\mathcal F_0} |v(t,\cdot)|^2 \right]
+ \int_{\mathcal F_0} |D(v)|^2 =\dfrac{1}{2} \left[ m |\ell(0)|^2 + \mathcal{J}\omega(0)^2 + \int_{\mathcal F_0} |v(0,\cdot)|^2 \right] 
$$
The results are extended to $L^q$-initial data in \cite{Ervedoza_Hillairet__disk_fluid}.

\medskip

\medskip

Solutions to the Cauchy-problem are constructed via a perturbative approach. First the authors consider the linearized system:
\begin{align}
    \frac{\partial v}{\partial t}  - \Delta v + \nabla \tilde{p} = 0& \quad \textnormal{for } t \in (0, \infty), x \in \mathcal{F}_0, \label{eq:Stokes_modfirst} \\
    \divg v = 0& \quad \textnormal{for } t \in (0, \infty), x \in \mathcal{F}_0, \\
    v (t, x) = \ell (t) + \omega (t) \, x^\perp& \quad \textnormal{for } t \in (0, \infty), x \in \partial B_0, \\
    m \ell'(t) = - \int_{\partial B_0} \Sigma (v,\tilde{p}) n \diff \sigma (x)& \quad \textnormal{for } t \in (0, \infty), \\
    \mathcal{J} \omega' (t) = - \int_{\partial B (t)} x^\perp \cdot \Sigma (v,\tilde{p}) n \diff \sigma (x)& \quad \textnormal{for } t \in (0, \infty),  \label{eq:Stokes_modlast}
 \end{align}
 They show that this system can be rewritten into an infinite-dimensional 
 differential system 
\begin{equation} \label{eq_HY}
 {
\partial_t V + AV = 0,}
\end{equation}
by  constructing an unknown $V$ encoding simultaneously $(v,\ell,\omega)$
and a specific unbounded operator $A$ (that we call fluid-strucure operator following \cite{DMS}). We give more details on these constructions in the next section. Finite-energy solutions to the linearized system are obtained by remarking that $A$ is an accretive positive self-adjoint operator which implies the existence of a contraction semi-group $(S(t))_{t>0}$ solving \eqref{eq_HY}. The nonlinear system can be then interpreted in the form of a nonlinear infinite-differential system:
$$
{ \partial_t V + A V = F(V),}
$$ 
and mild-solutions are constructed via a Kato-type argument. Since these mild-solutions are finite-energy solutions and finite-energy solutions are unique, this yields "the" finite-energy solution. Actually, this argument is performed on regularized $H^1$ initial data  in \cite{tucsnaktakahashi2004} (and finite-energy solutions are obtained then by a compactness argument). But, as we shall see below (see {\bf Theorem \ref{th:cauchy}}), the reasoning extends to $L^2$
initial data. { An alternative approach relying on Leray-type arguments is also provided in \cite{Bravin} in case of Navier-type slip boundary conditions on the fluid/solid interface.  Since our results rely strongly on decay estimates of the semi-group $A$ we shall stick to this mild-solution approach herein.}

\medskip

The long-time behavior of solutions to \eqref{eq:NS_mod1}-\eqref{eq:NS_mod7} is tackled in \cite{Ervedoza_Hillairet__disk_fluid} by the second author in collaboration with S. Ervedoza and C. Lacave.  Firstly the properties of the fluid-structure semi-group $(S(t))_{t \geq 0}$ are studied in a non-Hilbert setting which yields explicit bounds for the large-time decay of $L^q$-initial data and an explicit first order term for sufficiently localized initial data. {\em Via} perturbative arguments, these decay rates are extended to the  finite-energy solutions to  the full nonlinear problem \eqref{eq:NS_mod1}--\eqref{eq:NS_mod7} for initial data such that
 $v_0 \in L^q(\mathcal F_0) \cap L^2(\mathcal F_0)$
for some $q\in (1,2)$ with
$\|v_0\|_{L^2(\mathcal F_0)} + |\ell_0| + |\omega_0| $ sufficiently small (depending only on $q$).

\medskip

In this paper we pursue the computations of \cite{Ervedoza_Hillairet__disk_fluid} in two directions. Firstly, we extend the decay-rate computation of finite-energy solutions to \eqref{eq:NS_mod1}--\eqref{eq:NS_mod7} for arbitrary data in $L^q \cap L^2$. Namely, our first result reads:
\begin{theorem} \label{thm_main3}
Let $q \in (1,2)$ and assume that $t_0=0$ and that the initial data $(v_0,\ell_0,\omega_0) \in L^2(\mathcal F_0) \times \mathbb R^2 \times \mathbb R$ satisfy the compatibility condition \eqref{eq_compatinitdata}  and the further condition $v_0 \in L^q(\mathcal F_0).$ Then, the unique finite-energy solution  $(v,\ell,\omega)$ of  \eqref{eq:NS_mod1}-\eqref{eq:NS_mod7}  satisfies:
\begin{align}
&  \sup_{t >0} t^{\frac 1p - \frac 1q}\|v(t)\|_{L^p(\mathcal F_0)} <\infty \quad \forall \,p \in (2,\infty) \label{eq:decay1}\\
& \sup_{t >0} t^{\frac 1q}|\ell(t)| < \infty   \label{eq:decay2}
\end{align}
\end{theorem}
  
 This result must be compared with \cite[Theorem 1.3]{Ervedoza_Hillairet__disk_fluid} where a further smallness is required. We point out that, like in \cite{Ervedoza_Hillairet__disk_fluid}, our result states that  the decay of solutions to the nonlinear problem \eqref{eq:NS_mod1}-\eqref{eq:NS_mod7} is the same as the decay of solutions of the linearized system \eqref{eq:Stokes_modfirst}-\eqref{eq:Stokes_modlast}. However, we are still not able to extract a leading term for the nonlinear system.
 
 \medskip
 
 The proof of this first result is based on adapting the global stability argument in \cite{Gallay_Maekawa__NS_obstacle}. Namely, we use that the fluid-structure operator $A$ underlying the resolution of the linearized problem \eqref{eq:Stokes_modfirst}-\eqref{eq:Stokes_modlast} is self-adjoint and positive. We can then construct the 
 fractional powers $A^{\mu}$ for $\mu \in (-1,1)$ and analyze their ranges and domains. To extract a decay of any solution to \eqref{eq:NS_mod1}-\eqref{eq:NS_mod7}, we first compute an energy estimates on $U = A^{-\mu} V$
 for a $\mu$ adapted to the integrability of the initial data $v_0.$
One key new difficulty is that the nonlinearities in \eqref{eq:NS_mod1}--\eqref{eq:NS_mod7} involve the term $\ell \cdot \nabla v.$ It turns out that handling this term requires to prove a similar time-integrability of $\ell$ as the one of $\nabla v$ and in particular that 
$\ell \in L^2((0,\infty)).$ This property is obtained in a first independent step.
  
\medskip

In a second direction, we also extend the analysis to infinite energy initial data. Indeed, similarly to the introductory remark of \cite{Gallay_Maekawa__NS_obstacle} in the case of a still particle, one may observe that  the total amount of the fluid vorticity $\omega := \partial_2 v_1 - \partial_1 v_2$ in solutions to \eqref{eq:NS_mod1}-\eqref{eq:NS_mod6} has to vanish. This property fails however in many contexts. We recall that, in the absence of a disk, a central object is  the normalized Lamb-Oseen vortex: 
\begin{equation} \label{eq:Oseen_vort}
    \Theta (t,x) = \frac{1}{2 \pi} \frac{x^\perp}{\abs{x}^2} \Bigl( 1 - e^{- \frac{\abs{x}^2}{4 (1+t)}} \Bigr), \qquad x \in \mathbb{R}^2 \setminus \{ (0,0) \}, \quad t \geq 0,
\end{equation}
since any solution to the Navier Stokes equations  on $\mathbb R^2$ converges to a multiple of this profile  given by the initial mass of the vorticity \cite{GallayWayne}. This result is extended to the Navier Stokes equations outside a still obstacle \cite{Gallay_Maekawa__NS_obstacle} showing that any
bounded-energy perturbation of a small Lamb-Oseen vortex behaves in large-time like the Lamb-Oseen vortex. 

\medskip

We consider herein the local stability of the Lamb-Oseen vortex $\Theta$ in the case of the full fluid+disk problem \eqref{eq:NS_mod1}--\eqref{eq:NS_mod6}. For this, we first see that $\Theta$ can be written under the form $\Theta (t,x) = g(t,\abs{x}^2) \, x^\perp$, where
\begin{equation*}
    g(t,r) = \frac{1 - e^{- \frac{r}{4 (1+t)}}}{2 \pi r}.
\end{equation*}
Hence, the Lamb-Oseen vortex on $\partial B_0$ is a pure rotation. 
We can then  assume initial data are of the form
\begin{equation}
    v_0 = \alpha \Theta(t_0,\cdot) + w_0  \qquad \ell_0 = \ell^0_{w}  \qquad \omega_0 = \dfrac{\alpha}{2\pi} \left( 1 - \exp(-1/4(1+t_0)) \right) + \omega^{0}_{w}
    \label{eq_initdatw0}  
\end{equation}
where $w_0$ is localized in space  and 
\begin{equation} \label{eq_compat}
w_0 = \ell^0_w + \omega_w^0  
x^{\bot}  \text{ on $\partial B_0$}.
\end{equation}
Furthermore, we remark (or recall) that the Lamb-Oseen vortex yields a solution to 
the Navier Stokes equations with an explicit pressure:
\begin{equation*}
\nabla \Pi = \alpha^2 \dfrac{x}{|x|^2} |\Theta(t,x)|^2 \quad \forall \, x \in \mathbb R^2 \setminus \{(0,0)\}.
\end{equation*}
Hence, plugging the ansatz:
\begin{equation} \label{eq:ansatz}
\begin{aligned}
    & v (t,x) = \alpha \Theta (t,x) + w (t,x), && 
    \tilde{p} (t,x) = \alpha^2 \Pi (t,x) + q (t,x), \\
    & \ell_v (t,x) = \ell_w (t,x) && 
    \omega_v (t,x) = \alpha g(t, 1) + \omega_w (t,x).
    \end{aligned}
\end{equation}
into \eqref{eq:NS_mod1}-\eqref{eq:NS_mod7}, we obtain the perturbed system:
\begin{align}
    \frac{\partial w}{\partial t} + ((w - \ell_w (t)) \cdot \nabla) w - \Delta w + {\nabla { q}}& = - \alpha \Bigl[ (\Theta \cdot \nabla) w + ((w - \ell_w (t)) \cdot \nabla) \Theta \Bigr]&& \textnormal{in } (0, \infty) \times \mathcal{F}_0, \label{eq:init_prob1} \\
    \divg w&  = 0&& \textnormal{in } (0, \infty) \times  \mathcal{F}_0, \\
    w (t, x) & = \ell_w (t) + \omega_w (t) x^\perp&& \textnormal{on }  (0, \infty) \times \partial B_0, \label{eq:init_prob3} \\
    m \ell_w'(t)& = - \int_{\partial B_0}{ \Sigma (w,{ q})} n \diff \sigma (x)&& \textnormal{on } (0, \infty), \\
    \mathcal{J} \omega_w' (t) &= - \int_{\partial B (t)} x^\perp \cdot { \Sigma (w,{ q})} n \diff \sigma (x) + \alpha \zeta (t)&& \textnormal{on } (0, \infty), \label{eq:init_prob5} \\
    w|_{t=0} &= w_0&& \textnormal{on }  \mathcal{F}_0, \\
    \ell_w (0) & = \ell_w^0, \, \omega_w (0) = \omega_w^0 .& &\label{eq:init_prob7}
\end{align}
with an explicit source term $\zeta.$ We detail this computation in {\bf Section \ref{sec:Cauchy}}. We can then rely on the study of the 
fluid-structure semi-group to construct a mild-solution 
to \eqref{eq:init_prob1}--\eqref{eq:init_prob7}:
\begin{equation} \label{eq:Duhamel}
W(t) = S(t-t_0) W_0 + \int_{t_0}^{t} S(t-s) F_{\alpha}(s){\rm d}s
\end{equation}
with a source term $F_{\alpha}$ to be made precise later on.

\medskip

In this  direction, our first result shows that this Duhamel-formula yields a suitable solution to our problem: 
\begin{theorem} \label{thm_main1}
Let $(\alpha,t_0) \in \mathbb R \times [0,\infty)$ and { $(w_0, \ell^0_w,\omega_w^0) \in L^2(\mathcal F_0) \times \mathbb R^2 \times \mathbb R$ such that \eqref{eq_compat} is satisfied.} Then,  the Duhamel formula \eqref{eq:Duhamel} yields a triplet $(w,\ell_w,\omega_w)$ such that:
\begin{enumerate}
\item { $w \in C([t_0,\infty);L^2(\mathcal F_0)) \cap C((t_0,\infty);H^1(\mathcal F_0)),$ with $w \in L^2_{loc}((t_0,\infty);H^2(\mathcal F_0))$}
\item { $(\ell_w,\omega_w) \in C^1([t_0,\infty); \mathbb R^3)$}
\item $(w,\ell_w,\omega_w)$ is a solution to \eqref{eq:init_prob1}--\eqref{eq:init_prob7} 
\end{enumerate}
\end{theorem}

By reconstructing $(v,\ell,\omega)$ via \eqref{eq:ansatz}, we recover a global-in-time solution for unbounded-energy initial data of the form \eqref{eq_initdatw0}. We can then look at the large-time behavior of these solutions. To state this second result we shall start from a sufficiently developed Lamb-Oseen vortex, meaning that the radius of the vortex is sufficiently large, or that we consider the problem \eqref{eq:NS_mod1}-\eqref{eq:NS_mod7} starting  from a time $t_0$ sufficiently large with an initial data obtained by perturbing $\alpha \Theta(t_0,\cdot)$ like in \eqref{eq_initdatw0} with a small perturbation in $\mathcal{L}^2$.
We have then the following theorem:
\begin{theorem} \label{thm:main2}
{  Let $(\alpha,t_0) \in \mathbb R \times [0,\infty)$ and $(w_0, \ell^0_w,\omega_w^0) \in L^2(\mathcal F_0) \times \mathbb R^2 \times \mathbb R$ such that \eqref{eq_compat} is satisfied}.  Assume further that $t_0$ is sufficiently large, $\alpha$ is sufficiently small, $w_0 \in L^q(\mathcal F_0)$ for some $q \in (1,2)$ and {$(w_0, \ell^0_w,\omega_w^0)$ is small enough in $L^2(\mathcal F_0) \times \mathbb R^2 \times \mathbb R.$} The constructed solution $(v,\ell,\omega)$ to \eqref{eq:NS_mod1}--\eqref{eq:NS_mod7} with initial condition \eqref{eq_initdatw0} satisfies:
\begin{align} \label{eq_result_unb1}
&  \lim_{t \to \infty} t^{\frac 12 - \frac 1p}\|v(t) - \alpha \Theta(t,\cdot)\|_{L^p(\mathcal F_0)} = 0 \quad \forall \,p \in (2,\infty) \\
& \sup_{t >t_0} (t-t_0)^{\frac 1q} |\ell(t)| < \infty   \label{eq_result_unb2}
\end{align}
\end{theorem}

Some comments are in order. First, the decay rate prescribed in \eqref{eq_result_unb1} implies that $\alpha \Theta$ is indeed the leading term for large times. However, the explicit formula \eqref{eq:Oseen_vort} entails that we have  $|\Theta(t,x)| \leq 1/t$ on $\partial B_0$ so that the remainder may be much larger { on $\partial B_0$} and induce a leading translation velocity. The complementary inequality \eqref{eq_result_unb2} fixes then a minimal decay of the translation velocity depending only on the integrability of the initial perturbation. 

\medskip

The proofs of the two latter theorems rely on the $L^p-L^q$ properties of the semi-group $(S(t))_{t\geq 0}$  obtained in \cite{Ervedoza_Hillairet__disk_fluid}. 
One key-difficulty in both cases is again the term $\ell_w \cdot \nabla w.$ This term has limited space integrability (we cannot expect better than $\nabla w \in L^2(\mathcal F_0)$) and time-decay ($|\ell_w|$ decays a little less than $\|\nabla w\|_{L^2(\mathcal F_0)}$ but strictly less {\em a priori}). Hence, to handle this term we have to estimate sharply the loss of time-decay  between $|\ell_w|$ and $\|\nabla w\|_{L^2(\mathcal F_0)}.$ This is obtained by applying a sharp version of the Galgliardo-Nirenberg inequality and of the associated constant, following \cite{DelPino_Dolbeaut__GN_const}.

\medskip

The outline of the paper is as follows. In the next section we provide preliminary lemmas. We explain the construction of the capital-letter unknowns and fluid-structure operator $A.$ We recall the results of  \cite{Ervedoza_Hillairet__disk_fluid} on the decay properties of the semi-group and complement the analysis with a descrpition of the fractional powers of $A$ in the spirit of \cite{Gallay_Maekawa__NS_obstacle}. Finally, we recall the Gagliardo Nirenberg analysis underlying the stability analysis of the Lamb-Oseen vortex. In Section \ref{sec:Cauchy} we detail the proofs of Theorem \ref{thm_main1} and Theorem \ref{thm:main2}. Section \ref{sec:bounded} is devoted to the proof of Theorem \ref{thm_main3}. Some further technicalities are presented in an appendix.


\section{Preliminary constructions and technical lemmas}
In this section, we first recall the construction of function spaces that enable to handle  the fluid unknown $v$ and solid unknowns $(\ell,\omega)$ at once.
We also recall the construction of the unbounded operator $A$ underlying the resolution of \eqref{eq:Stokes_modfirst}-\eqref{eq:Stokes_modlast}. 
These constructions are reproduced from  \cite{Ervedoza_Hillairet__disk_fluid,tucsnaktakahashi2004,WangXin2011}. 

\medskip

The first key-issue we address  is related to the problem of controlling the body linear velocity by the fluid velocity-field. In the forthcoming analysis, one would hope to be able to control the linear velocity $|\ell|$  by  $\|\nabla v\|_{L^2(\mathcal F_0)}$ only. However, in full generality,   this is possible in $3D$ but it turns out to be false in $2D.$ This can be seen as reminiscent either of the fact that $\dot{H}^1(\mathbb R^2)$ embeds in no $L^p(\mathbb R^2)$ space or of the Stokes paradox \cite[Introduction of Section V]{Galdibooknew}.
 Here, we exchange such a control for an almost optimal control in the form of a family of Gagliardo-Nirenberg inequalities with an explicit estimate of the embedding constants.
The second key-contribution of this section is the analysis of the 
fractional powers of the operator $A.$

\subsection{Function spaces and Gagliardo-Nirenberg inequality}
As classical in fluid+disk systems, we treat  \eqref{eq:NS_mod1}-\eqref{eq:NS_mod7} by encoding all the unknowns $(v,\ell_,\omega)$ into one unified unknown with the following construction. 
From a triplet $(v, \ell, \omega) \in [\mathcal C^{\infty}_c(\overline{\mathcal F_0})]\times \mathbb R^2 \times \mathbb R$ verifying 
\begin{equation*}
\divg v = 0 \quad \textnormal{in } \mathcal{F}_0, \qquad v = \ell + \omega \, x^\perp \quad \textnormal{ on } \partial B_0,
\end{equation*}
we define a divergence-free vector field denoted $V$ on $\mathbb{R}^2$ obtained by extending $v$ by $\ell + \omega \, x^\perp$ in $B_0$. Adapted to such $V$, we introduce the function spaces $\mathcal{L}^p$ ($p\in[1,\infty]$) defined by
\begin{equation*}
    \mathcal{L}^p \coloneqq \{ V \in [L^p (\mathbb{R}^2)]^2, \, \divg V = 0 \textnormal{ in } \mathbb{R}^2, \, D(V) = 0 \textnormal{ in } B_0 \}.
\end{equation*}
We recall that, since $B_0$ is connected,  the condition $D(V)=0$ on $B_0$ implies that $V_{|_{B_0}}$ is a rigid velocity-field.
Conversely, we adapt below the convention that for $V \in \mathcal L^{p}$ we denote $v = \mathds{1}_{\mathcal F_0}V$ and $(\ell_{v},\omega_{v}) \in \mathbb R^2 \times \mathbb R$ the translation/angular velocities characterizing $V$ in $B_0.$

\medskip

We recall now some classical properties of these spaces. When $p \in [1, \infty)$, we endow $\mathcal L^p$ with the norm
\begin{equation*}
    \norm{V}_{\mathcal{L}^p}^p = \int_{\mathcal{F}_0} \abs{V}^p + \frac{m}{\pi} \int_{B_0} \abs{V}^p,
\end{equation*}
(and the corresponding definition when $p = \infty$).When $(p,p') \in [1,\infty]$ are conjugate, we equip $(\mathcal L^p,\mathcal L^{p'})$ with the duality pairing:
\begin{equation*}
    \langle V, W \rangle_{\mathcal{L}^p, \mathcal{L}^{p'}} = \int_{\mathcal{F}_0} V \cdot W + \frac{m}{\pi} \int_{B_0} V \cdot W.
\end{equation*}
For any $p\in [1,\infty],$ it is straightforward that $\mathcal L^{p}$ is a closed subspace of 
$$
L^{p}_{\sigma}(\mathbb R^2) := \{ V \in L^p(\mathbb R^2) \text{ s.t. } {\rm div } V = 0 \} 
$$
which is itself a closed subspace of $[L^p(\mathbb R^2)]^2.$ In particular, there exists a projector $\mathbb P_p: [L^p(\mathbb R^2)]^2 \to \mathcal L^p.$ When $p \in (1,\infty),$ this projector is analyzed in previous references such as \cite{WangXin2011}. Since all the $\mathbb P_{p}$ coincide on $C^{\infty}_c(\mathbb R^2)$ we can drop the $p$-dependency and denote this projector with $\mathbb P.$ Our analysis below relies on the following fundamental lemma whose proof can be found in \cite[Remark 2.4]{WangXin2011}:
\begin{lem}
Given $p \in (1,\infty)$ the projector $\mathbb P : [L^p(\mathbb R^2)]^2 \to \mathcal L^p$ is bounded. 
\end{lem}

\medskip
We also define
\begin{equation*}
    \mathcal{H}^1 \coloneqq \mathcal{L}^2 \cap { [H^1 (\mathbb{R}^2)]^2}.
\end{equation*}
As a closed subspace of $[H^1(\mathbb R^2)]^2$ this is a separable Hilbert space when equipped  with the norm
\begin{equation*}
    \norm{V}_{\mathcal{H}^1} = \norm{V}_{\mathcal{L}^2} + \norm{\nabla V}_{L^2},
\end{equation*}
in which the set of $C^{\infty}_c(\mathbb R^2)$-soleonidal vector-field  is dense. Implicitly in the gradient norm, we use
the shortcut $L^2$ for $L^2(\mathbb R^2).$ We keep this convention for norms of Lebesgue and Sobolev spaces in what follows.
The $\mathcal H^1$-norm is associated with a Korn inequality that reads as follows:
\begin{lem} \label{lem:grad_sym}
For any $V \in \mathcal H^1$ there holds:
\begin{equation} \label{eq:grad_sym}
\int_{\mathbb R^2} |\nabla V|^2 = 2 \int_{\mathcal F_0} |D(V)|^2.
\end{equation}
\end{lem}
We refer to \cite[Lemma 4.1]{tucsnaktakahashi2004} for a proof.

\medskip

We complement this part of the section with a Gagliardo-Nirenberg inequality that will enable to control the linear velocity associated with a fluid velocity-field. 
{
We build on the following result of \cite{DelPino_Dolbeaut__GN_const}:  
    \begin{lem}[{\cite[Theorem~1.1]{DelPino_Dolbeaut__GN_const}}]
        Let $d \geq 2$ and $q \geq 1$ such that $q \leq \frac{d}{d-2}$ if $d \geq 3$. Define
        \begin{equation*}
            \mathcal{D}^q (\mathbb{R}^d) = \{ u \in L^{q+1} (\mathbb{R}^d) \cap L^{2q} (\mathbb{R}^d) \, | \, \nabla u \in L^2 (\mathbb{R}^d) \}
        \end{equation*}
        Then, for any function $u \in \mathcal{D}^q (\mathbb{R}^d)$, there holds
        \begin{equation*}
            \norm{u}_{L^{2q}} \leq A_{q,d} \norm{\nabla u}_{L^2}^\theta \norm{u}_{L^{q+1}}^{1-\theta},
        \end{equation*}
        where
        \begin{equation*}
            A_{q,d} \coloneqq \biggl( \frac{y (q-1)^2}{2 \pi d} \biggr)^{\frac{\theta}{2}} \biggl( \frac{2 y - d}{2y} \biggr)^\frac{1}{2q} \biggl( \frac{\Gamma (y)}{\Gamma (y - \frac{d}{2})} \biggr)^\frac{\theta}{d},
        \end{equation*}
        with
        \begin{equation*}
            \theta = \frac{d (q-1)}{q (d + 2 - (d-2) q)}, \qquad
            y = \frac{q+1}{q-1}.
        \end{equation*}
    \end{lem}
And we obtain the following lemma:   
}
    \begin{lem} \label{lem:GN_const}
        There exists $C > 0$ such that, for any $p \geq 2$ and any $u \in H^1 (\mathbb{R}^2)$, there holds
        \begin{equation*}
            \norm{u}_{L^p} \leq C \sqrt{p} \, \norm{u}_{L^2}^{\frac{2}{p}} \norm{\nabla u}_{L^2}^{1-\frac{2}{p}}.
        \end{equation*}
    \end{lem}
    
   \begin{proof}

     Let $u \in H^1 (\mathbb{R}^2)$. Applying the previous lemma with $d=2$ and $q = \frac{p}{2}$, we get
    \begin{equation} \label{eq:est_DPB}
        \norm{u}_{L^{p}} \leq A_{q,2} \norm{\nabla u}_{L^2}^\theta \norm{u}_{L^{q+1}}^{1-\theta},
    \end{equation}
    with 
    \begin{equation*}
        \theta = \frac{q-1}{2q} = \frac{1}{2} - \frac{1}{p}, \quad
        y = \frac{q + 1}{q - 1} = 1 + \frac{4}{p - 2},
    \end{equation*}
    and
    \begin{equation*}
        A_{q, 2} = \biggl( \frac{y (q-1)^2}{4 \pi} \biggr)^{\frac{\theta}{2}} \biggl( \frac{y - 1}{y} \biggr)^\frac{1}{p} \biggl( \frac{\Gamma (y)}{\Gamma (y - 1)} \biggr)^\frac{\theta}{2}.
    \end{equation*}
    Using the property of the Gamma function, we have $\Gamma (y) = (y-1) \, \Gamma (y-1)$, so that
    \begin{equation} \label{eq:est_A}
        A_{q,2} = \biggl( \frac{(p+2) (p-2)}{16 \pi} \biggr)^{\frac{1}{4} - \frac{1}{2p}} \biggl( \frac{4}{p+2} \biggr)^\frac{1}{p} \biggl( \frac{4}{p-2} \biggr)^{\frac{1}{4} - \frac{1}{2p}} \leq C \, p^\frac{1}{4}.
    \end{equation}
    Moreover, by interpolation, there holds
    \begin{equation*}
        \norm{u}_{L^{q+1}} \leq \norm{u}_{L^2}^\frac{2}{p+2} \norm{u}_{L^{p}}^{\frac{p}{p+2}}.
    \end{equation*}
    Thus, putting this and \eqref{eq:est_A} into \eqref{eq:est_DPB} yields
    \begin{equation*}
        \norm{u}_{L^{p}} \leq C \, p^\frac{1}{4} \, \norm{\nabla u}_{L^2}^{\frac 12 - \frac 1p}\norm{u}_{L^2}^{\frac{1}{p}} \norm{u}_{L^{p}}^{\frac{1}{2}}.
    \end{equation*}
    The conclusion follows.
\end{proof}
 
The above lemma entails the following control that we shall use without mention below:
\begin{cor} \label{cor:est_l_GN}
    Let $p \geq 2$ and $V \in \mathcal L^p \cap H^1(\mathbb R^2).$ There exists a constant $C$ independent of $p$
    and $V$ such that:
    \begin{equation*}
    |\ell_v| \leq C \sqrt{p} \|V\|_{\mathcal L^2}^{\frac 2p} \|\nabla V\|_{L^2(\mathbb R^2)}^{1 - \frac 2p}.
    \end{equation*}
\end{cor}

\subsection{Construction of the unbounded operator $A$ and related properties}
With the construction of the previous part in this section, we can now define the
fluid-structure operator $A$ which enables to rewrite the system  \eqref{eq:Stokes_modfirst}-\eqref{eq:Stokes_modlast} into the infinite-dimensional differential system \eqref{eq_HY}. Following \cite{Ervedoza_Hillairet__disk_fluid,tucsnaktakahashi2004,WangXin2011}
we set:
$$
\mathcal D(A) :=\{W \in \mathcal H^1(\mathbb R^2) \text{ s.t. } w = W_{|_{\mathcal F_0}} \in [H^2(\mathcal F_0)]^2\}.
$$
We point out that such vector-fields admit a discountinuity of normal derivative on $\partial B_0.$ 
This is a key property that enables a non-trivial solid dynamics. For any $W \in \mathcal D(A)$
we set $A W = \mathbb P\mathcal AW$ where (keeping the convention that  $w = W_{|_{\mathcal F_0}}$)
$$
\mathcal AW =
\{
\begin{aligned}
& -  \Delta  w & \text{ in $\mathcal F_0$}\\
& \dfrac{2}{m} \left( \int_{\partial B_0} D(w) n {\rm d}\sigma\right)  + 2 \mathcal J^{-1} \left(  \int_{\partial B_0}
z^{\bot} D(w) n {\rm d} \sigma \right) y^{\bot} & \text{ in $B_0$}.
\end{aligned}
\right.
$$
We note that this induces indeed an unbounded operator $\mathcal D(A) \to \mathcal L^{2}(\mathbb R)$ because
for any $W \in \mathcal D(A)$ we have $\mathcal AW \in [L^{2}(\mathbb R^2)]^2$ (so that in particular $\mathbb P$
corresponds actually to the $L^2$-projection). 

\medskip

\subsubsection{Previous analysis of $A$}
In \cite{tucsnaktakahashi2004} the properties of  $A$ are studied in this hilbertian framework.
We gather here the main conclusions. First, we have that the  unbounded operator $(A,\mathcal D(A))$ is an accretive self-adjoint positive operator on $\mathcal L^2$. Hence, the Cauchy problem
\begin{equation} \label{eq:CPabstrait}
\{
\begin{aligned}
& \partial_t V + A V = 0 \\
& V_{|_{t=0}} = V_0
\end{aligned}
\right.
\end{equation}
has a unique solution  for any $V_0 \in \mathcal L^2$  defining thus a contraction semi-group $(S(t))_{t>0}.$ The relations between this semi-group and our linearized system is the content of the following proposition:  
\begin{prop}
For  any $V_0 \in \mathcal L^2$, the unique solution 
$$
V \coloneqq S(t) V_0 \in C([0,\infty); \mathcal L^2) \cap C^1((0,\infty) ; \mathcal L^2) \cap C((0,\infty);\mathcal D(A))
$$
to the Cauchy problem \eqref{eq:CPabstrait}
yields a vector field $v$ and velocities $(\ell_v,\omega_v)$ satisfying
\begin{itemize}
\item $v \in C([0,\infty);L^2(\mathcal F_0)) \cap C((0,\infty) ; H^2(\mathcal F_0))$,
\item $(\ell_v,\omega_v) \in C([0,\infty); \mathbb R^2 \times \mathbb R) $,
\end{itemize}
and a pressure $p \in C((0,\infty); H^1_{loc}(\mathcal F_0))$ such that \eqref{eq:Stokes_modfirst}-\eqref{eq:Stokes_modlast} holds true with initial condition:
$$
\ell_v(0) = \ell_0, \quad \omega_v(0) = \omega_0, \quad v(0,\cdot) = v_0 \quad \text{ in $\mathcal F_0$}. 
$$
\end{prop}

Remarking that the spaces $(\mathcal L^p)_{p \in (1,\infty)}$ share $\mathcal L^2 \cap C^{\infty}_c(\mathbb R^d) $ as dense subspace the properties of the semi-group $(S(t))_{t>0}$ are extended to the non-hilbertian setting in \cite{Ervedoza_Hillairet__disk_fluid}. This is the content of the following lemma:

\begin{lem}[{\cite[Theorem~1.1]{Ervedoza_Hillairet__disk_fluid}}] \label{lem:est_semigroup}
    For each $q \in (1, \infty)$, the fluid-structure operator $A$ generates a semi-group on $\mathcal{L}^q$ which satisfies:
    \begin{itemize}
    \item For all $p \in [q, \infty]$, there exists $K_1 = K_1 (p,q) > 0$ such that for every $V_0 \in \mathcal{L}^q$:
    \begin{equation*}
        \norm{S(t) V_0}_{\mathcal{L}^p} \leq K_1 t^{\frac{1}{p} - \frac{1}{q}} \norm{V_0}_{\mathcal{L}^q} \qquad \textnormal{for all } t > 0.
    \end{equation*}

    \item If $q \leq 2$,  for $p \in [q, 2]$, there exists $K_2 = K_2 (p, q) > 0$ such that for every $V_0 \in \mathcal{L}^q$:
    \begin{equation*}
        \norm{\nabla S(t) V_0}_{L^p (\mathcal{F}_0)} \leq K_2 t^{- \frac{1}{2} + \frac{1}{p} - \frac{1}{q}} \norm{V_0}_{\mathcal{L}^q} \qquad \textnormal{for all } t > 0.
    \end{equation*}

    \item For $p \in [\max (2, q), \infty)$, there exists $K_3 = K_3 (p, q) > 0$ such that for every $V_0 \in \mathcal{L}^q$:
   $$
        \norm{\nabla S (t) V_0}_{L^p (\mathcal{F}_0)} \leq 
            \begin{cases}
           K_3 t^{- \frac{1}{2} + \frac{1}{p} - \frac{1}{q}} \norm{V_0}_{\mathcal{L}^q} & \textnormal{for all } 0 < t < 1, \\
                K_3 t^{- \frac{1}{q}} \norm{V_0}_{\mathcal{L}^q} & \textnormal{for all } t \geq 1.
            \end{cases}
    $$
     \end{itemize}
\end{lem}
The above estimates for the gradient are only on $\mathcal{F}_0$. However, when $V_0 \in \mathcal L^2,$ $V(t) = S(t) V_0$ is in $\mathcal{H}^1$ (since it is in $\mathcal D(A))$ for $t >0$ so that Lemma \ref{lem:grad_sym} applies.  Thus, the estimates in Lemma \ref{lem:est_semigroup} are sufficient to get a full $\mathcal{H}^1$ estimate.
Last, we also recall duality decay estimates as shown in \cite{Ervedoza_Hillairet__disk_fluid}.

\begin{lem}[{\cite[Corollaries~3.10~and~3.11]{Ervedoza_Hillairet__disk_fluid}}] \label{lem:duality_est}
    Assume $1 < q \leq p < \infty$ and let $F \in L^q (\mathbb{R}^2; M_2 (\mathbb{R}))$ satisfying $F = 0$ on $B_0$. The following decay estimates for $V(t) = S(t) \mathbb{P} \divg F$ hold true:
    \begin{itemize}
    \item if $q \geq 2$, there exists $K_4 = K_4 (p, q) > 0$ such that for all $t > 0$:
    \begin{equation*}
        \norm{V (t)}_{\mathcal{L}^p} \leq K_4 \, t^{-\frac{1}{2} + \frac{1}{p} - \frac{1}{q}} \norm{F}_{L^q (\mathbb{R}^2)}.
    \end{equation*}

    \item if $q \leq 2$, there exists $K_5 = K_5 (p,q) > 0$ such that:
    \begin{equation*}
        \norm{V(t)}_{\mathcal{L}^p} \leq
        \begin{cases}
            K_5 \, t^{-\frac{1}{2} + \frac{1}{p} - \frac{1}{q}} \norm{F}_{L^q (\mathbb{R}^2)}& \qquad \textnormal{for all } 0 < t < 1, \\
            K_5 \, t^{-1 + \frac{1}{p}} \norm{F}_{L^q (\mathbb{R}^2)}& \qquad \textnormal{for all } 1 \leq t.
        \end{cases}
    \end{equation*}

    \item There exists $K_\ell = K_\ell (q) > 0$ such that for all $t > 0$:
    \begin{equation*}
        \abs{\ell_{V(t)}} \leq K_\ell \, t^{-\frac{1}{2} - \frac{1}{q}} \norm{F}_{L^q (\mathbb{R}^2)}.
    \end{equation*}
    \end{itemize}
\end{lem}

\subsubsection{Further material on $A$}

In this part, we complement the analysis  of $A$ with more properties of its fractional powers.
The fluid-structure operator $A$ being self-adjoint and positive definite, we may define $A^{\mu}$ for $\mu \in (-1,1)$
through its spectral representation  \cite[Section II.3.2]{Sohr}. Since $A$ is injective, we have that these fractional powers (either positive or negative) are positive self-adjoint operators with dense domains. 

\medskip

Our first proposition concerns the square-root of $A.$
\begin{lem} \label{lem:A_div_bounded}
\begin{enumerate}
\item 
We have
$\mathcal D(A^{\frac 12}) = \mathcal H^1(\mathbb R^2)$ and 
\begin{equation} \label{eq:link_A_nabla}
\|A^{\frac 12} V\|_{\mathcal L^2} = \sqrt{2} \|D(v)\|_{L^2(\mathcal F_0)}.
\end{equation}
\item 
{
Let $F \in [C^{\infty}_c(\mathcal F_0)]^{2\times 2}$ then, 
$$
{\mathbb P}{\rm div} F  \in \mathcal D(A^{- \frac 12})  
\quad 
\text{ with }
\quad
\|A^{-\frac 12} \mathbb P {\rm div} F\|_{\mathcal L^2} \leq  \|F\|_{L^2(\mathcal F_0)}.
$$}
\end{enumerate}
\end{lem}

\begin{proof}
{
We refer to \cite[p. 63]{tucsnaktakahashi2004} for a proof of the first item.
As for the second item, we follow \cite{Gallay_Maekawa__NS_obstacle} and propose a proof based on the approach of \cite[Lemma III-2.6.1]{Sohr}.  Since $A^{-\frac 12}$ is self-adjoint, and because of the identities \eqref{eq:grad_sym} and \eqref{eq:link_A_nabla}, our proof reduces to obtaining the bound:
$$
|\langle \mathbb P {\rm div} F, A^{-\frac 12} w \rangle| \leq \|F\|_{L^2(\mathcal F_0)}\|\nabla A^{-\frac 12} w\|_{L^2(\mathbb R^2)} \quad 
\forall \, w \in \mathcal D(A^{-\frac 12})
$$
Let  $w \in \mathcal D(A^{-\frac 12})$ so that there exists $v \in \mathcal D(A^{\frac 12})$ for which $w = A^{\frac 12}v$ (and thus $v = A^{-\frac 12}w)$. We have then by definition of projectors $\mathbb P$ and integration by parts:
\begin{align*}
\langle \mathbb P {\rm div} F, A^{-\frac 12} w \rangle & =  \dfrac{m}{\pi} \int_{B_0} {\rm div} F \cdot v 
+ \int_{\mathcal F_0} {\rm div} F \cdot v \\
& =   -  \int_{\mathcal F_0} F :  \nabla v.
\end{align*}
We conclude with a standard Cauchy-Schwarz inequality 
}
\end{proof}

In the proof above, if we do not make further assumption on the support of $F$ and
take $w \in\mathcal D(A^{- \frac 12})$, the last identity yields:
$$
\langle \mathbb P {\rm div} F, A^{-\frac 12} w \rangle =
\left( \dfrac{m}{\pi} - 1 \right) \int_{\partial B_0} Fn \cdot v  - \int_{\mathcal F_0} F :  \nabla v  
$$
where:
$$
\int_{\partial B_0} Fn \cdot v {\rm d}\sigma  =  \int_{\partial B_0} Fn {\rm d}\sigma \cdot \ell_v + 
\int_{\partial B_0} Fn \cdot n^{\bot} {\rm d}\sigma \omega .  
$$
To relax the assumption on the support of $F$ we should be able to control this further term by $\|\nabla v\|_{L^2(\mathbb R^2)}.$ This implies to obtain the boundedness
of the mapping $v \mapsto \ell_v$ on $\mathcal D(A)$ endowed with the $\dot{H}^{1}(\mathbb R^2)$ topology.
However again, the Stokes paradox implies that this property does not hold true. 
{ With the above computations, we can extend $A^{-1/2}\mathbb P {\rm div}$ by density into a mapping $(L^2(\mathcal F_0))^{2\times 2} \to \mathcal L^2.$ For the further analysis,  we need to analyze the  relations that exists then between $A^{1/2}S(\tau) [A^{-1/2} \mathbb P{\rm div}]$ and $S(\tau)\mathbb P{\rm div}$ when $\tau >0.$ This is the content of the next corollary 

\begin{cor} \label{cor_zecor}
Let $F \in (L^2(\mathcal F_0))^{2 \times 2}$ such that ${\rm div F} \in (L^2(\mathcal F_0) + L^{4/3}(\mathcal F_0)))^2$ and $F \cdot n = 0$ on $\partial \mathcal F_0.$ For arbitrary $\tau >0,$  we have:
\begin{equation} \label{eq_identity_debase}
A^{1/2} S(\tau)[A^{-1/2}\mathbb P {\rm div} F] = S(\tau){\mathbb P}{\rm div} F. 
\end{equation}
\end{cor}

\begin{proof}
Since $F \cdot n = 0$ on $\partial \mathcal F_0$ we can  construct $F_n \in C^{\infty}_c(\mathcal F_0)$ 
by a dilation/truncation and mollifying argument (see also \cite[Theorem 1.3]{Temam} for instance) such that we have simultaneously
$$
F_n \to F \text{ in $[L^2(\mathcal F_0)]^{2\times 2}$},
\quad 
{\rm div} F_n \to {\rm div} F \text{ in $[L^2(\mathbb R^2) + L^{4/3}(\mathbb R^2)]^2 $},
$$
Since the identity \eqref{eq_identity_debase} holds true at the level of $F_n$ it extends to $F$ by letting $n$ go to infinity. 
\end{proof}
}

\medskip

We proceed with the analysis of the range of $A^{\mu}$ for $\mu \in (0,1/2)$ corresponding
to \cite[Lemma 5.1]{Gallay_Maekawa__NS_obstacle}. This is the content of the next lemma:

\begin{lem} \label{lem:fractional}
Let $q \in (1,2)$ and $\mu < 1/q - 1/2.$ For all $v \in \mathcal L^2(\mathbb R^2) \cap [L^{q}(\mathbb R^2)]^2$ there exists a unique $w \in \mathcal D(A^{\mu})$ such that $v= A^{\mu} w.$ Furthermore, there exists a constant $C = C(q,\mu)>0$ depending only on $q$ and $\mu$ for which $\|w\|_{\mathcal L^2} \leq C( \|v\|_{L^q(\mathbb R^2)} + \|v\|_{L^2(\mathbb R^2)}).$
\end{lem} 

We point out that, in this statement, the condition $v \in \mathcal L^2 \cap [L^q(\mathbb R^2)]^2$ reads also $v \in \mathcal L^2 \cap \mathcal L^q.$
What remains of this section is devoted to the proof of this result. We first remark that the proof of \cite[Lemma 5.1]{Gallay_Maekawa__NS_obstacle} yields from \cite[Lemma 2.2]{KozonoOgawa-decay}. So, our proof reduces mostly to check that the fluid-structure operator $A$ satisfies the key-properties necessary to reproduce the proofs of these latter lemmas (that were concerned initially with the standard Stokes operator { with homogeneous boundary conditions}). 
In comparison with these previous results, we have a loss in terms of the correspondence $q \to \mu$ and also in the control which involves the $L^2$-norm. In \cite{KozonoOgawa-decay} the authors obtain similar results with $\mu=1/q-1/2$ and a control with the $L^q$-norm only. It seems we might not get such optimal { bounds} in our case. But this will not depreciate the final result.

\medskip

In \cite{KozonoOgawa-decay}, the properties of the Stokes operator are analyzed on $\mathcal F_0$ when complemented with vanishing boundary conditions.  The main argument is performed on a Laplace system and divergence-free constraints are then handled {\em via} abstract Heinz-Kato arguments (see \cite[Lemma II.3.2.3, p. 100]{Sohr}). 
With our setting, this Laplace operator reads as follows. 
We set:
$$
L^2_0[B_0]   := \{ V \in [L^2(\mathbb R^2)]^2 \text{ s.t. } V = 0 \text{ on $B_0$}\} 
$$
and $\mathbb P_0 : [L^2(\mathbb R^2)]^2 \to L^2_0[B_0]$ the corresponding orthogonal projection.
Then, we define the operator $\tilde{\mathcal A}_0$ by 
$$
\mathcal D(\tilde{\mathcal{A}}_0) :=\{ V \in L^2_0[B_0] \text{ s.t. } v \in  [H^2(\mathcal F_0)]^2\}.
$$
with
$$
\tilde{\mathcal A}_0[V] =  \mathbb{P}_0 [- \mathds{1}_{\mathcal F_0}\Delta V], \quad \forall \, V \in \mathcal D(\tilde{\mathcal A}_0).
$$

To take advantage of the analysis of \cite{KozonoOgawa-decay} in order to study the fractional powers of $A,$ we propose to use the same Heinz-Kato argument to handle the divergence-free constraint and to focus on the remaining Laplace equation (completed with non-standard integral boundary conditions) with the help of $\tilde{\mathcal A}_0.$ The operator $\tilde{\mathcal A}_0$ will take hold of the PDE and we shall complement the analysis with a fine study of our non standard boundary conditions. To this end, we first rewrite the integral boundary conditions introduced by $A.$ 
This is the content of the following lemma: 

\begin{prop} \label{prop_newnewt}
Let $V \in \mathcal D(A)$ then there holds:
$$
\mathcal AV = \dfrac{1}{m} \left( \int_{\partial B_0} \partial_n v {\rm d}\sigma \right) +  \mathcal{J}^{-1} \left( \int_{\partial B_0} z^{\bot} \cdot \partial_n v {\rm d}\sigma + 2 \omega_v \right) y^{\bot}
\quad  \text{ on $B_0$}.
$$ 
\end{prop}

\begin{proof}
It is sufficient to prove that, for any $V \in \mathcal D(A)$ and any $(\ell,\omega) \in \mathbb R^2 \times\mathbb R$ there holds:
$$
\int_{\partial B_0} 2 D(v) n \cdot (\ell + \omega z^{\bot}){\rm d}\sigma
 =    \int_{\partial B_0} \partial_n v {\rm d}\sigma \cdot \ell +  \left( \int_{\partial B_0} z^{\bot} \cdot \partial_n v {\rm d}\sigma + 2 \omega_v \right) \omega.
$$

So, let $V \in \mathcal D(A).$ Given $(\ell,\omega) \in \mathbb R^2 \times \mathbb R$ let:
$$
W = \nabla^{\bot} \left[ \chi(y) \left( \ell \cdot y^{\bot} + \omega |y|^2 \right)\right]
$$
where $\chi \in C^{\infty}_c(\mathbb R^2)$ is fixed but arbitrary satisfying $\mathds{1}_{B_0} \leq \chi \leq 1.$ We note that with such conventions, there holds $W \in \mathcal D(A)$
with $\ell_{W} = \ell$ and $\omega_{W} = \omega.$ We have then
by integration by parts (using several times that $w,W$ and $v,V$ are divergence free):
\begin{align*}
\int_{\partial B_0} 2 D(v) n \cdot (\ell + \omega z^{\bot}){\rm d}\sigma
& =\int_{\partial B_0} 2D(v) n \cdot w {\rm} d \sigma \\
& =  \int_{\mathcal F_0} {\rm div} (2 D(v)) \cdot w  + 2 D(v) : D(w) \\
& = \int_{\mathcal F_0} \Delta v \cdot w +  \int_{\mathbb R^2} \nabla V : \nabla W  \\
& = \int_{\partial B_0} \partial_n v \cdot w{\rm d}\sigma + 2 \omega_v \omega \\
& = \int_{\partial B_0} \partial_n v \cdot (\ell + \omega z^{\bot}) {\rm d}\sigma+ 2 \omega_v \omega.
\end{align*}
The term $\omega_v \omega$ appearing on the fourth line is the contribution of the (skew-symmetric part of the) gradients $\nabla V$ and $\nabla W$ on $B_0.$
This ends the proof.
\end{proof}

Thanks to  {\bf Proposition \ref{prop_newnewt}} we can now rewrite the fluid-structure operator $A = \mathbb P \tilde{\mathcal A}$ where $\tilde{\mathcal A}$ is defined (without the divergence-free constraint)
by the formula:
$$
\tilde{\mathcal A}W =
\{
\begin{aligned}
& - \Delta  w & \text{in $\mathcal F_0$}\\
& \dfrac{1}{m} \left( \int_{\partial B_0} \partial_n w {\rm d}\sigma \right) + \mathcal{J}^{-1} \left( \int_{\partial B_0} z^{\bot} \cdot \partial_n v {\rm d}\sigma + 2 \omega_v \right) y^{\bot}
&   \text{in $B_0$},
\end{aligned}
\right.
$$
for $W \in \mathcal D(\tilde{\mathcal A}) =  L^2[B_0] \cap { [H^1(\mathbb R^2)]^2}  \cap [H^2(\mathcal F_0)]^2$. Here, we denote:
$$
L^2[B_0] = \{ W \in {[L^2(\mathbb R^2)]^2} \text{ s.t. } W = \ell_{W} + \omega_{W} y^{\bot} \text{ on $B_0$} \}.
$$ 
We may reproduce here classical computations to obtain that $\tilde{\mathcal A}$ is a selfadjoint positive
operator on $L^2[B_0]$ since it is associated with the quadratic form:
$$
\langle \tilde{\mathcal A}W , V \rangle = \int_{\mathbb R^2} \nabla W : \nabla V, \quad 
\forall \, (W,V) \in \mathcal{D}(\tilde{\mathcal{A}}).
$$ 
We point out that the duality bracket is still the one associated with the disk density.
In particular, we have that (note that $\nabla W$ is the skew-symmetric matrix associated with $\omega_W$ on $B_0$): 
\begin{equation} \label{eq:sqrt1}
\|\tilde{\mathcal A}^{\frac 12} W\|^2_{L^2[B_0]} = \int_{\mathbb R^2} |\nabla W|^2  \quad \forall \, W \in \mathcal{D}(\tilde{\mathcal A}^{\frac 12}).
\end{equation}
and, for $\lambda >0:$ 
\begin{equation} \label{eq:sqrt2}
\|(\tilde{\mathcal A} + \lambda)^{\frac 12} W\|^2_{L^2[B_0]} =  \int_{\mathbb R^2} |\nabla W|^2 + 
\lambda  \langle W , W \rangle \quad \forall \, W \in \mathcal{D}(\tilde{\mathcal A}^{\frac 12x}).
\end{equation}
{ We recall that similar identities hold  with the operator $A.$} Thanks to these two latter identities, we can reproduce the procedure of \cite[Lemma 2.2] {KozonoOgawa-decay} and the proof of Lemma \ref{lem:fractional} reduces to obtaining the following proposition:
\begin{prop} \label{prop:fractional}
Let $q \in (1,2)$ and {$\mu <  1/q - 1/2.$} For all $\varepsilon > 0$, there exists a mapping
$R_{\mu,\varepsilon}:  L^2[B_0] \cap [L^{q}(\mathbb R^2)]^2 \to L^2[B_0] $  satisfying:
\begin{itemize}
\item for arbitrary $W \in L^2[B_0]$ there holds:
$$
(\tilde{\mathcal A} + \varepsilon)^{-\mu} W  = (\tilde{\mathcal A}_0+ \varepsilon)^{-\mu}(\mathds{1}_{\mathcal F_0} W) + R_{\mu,\varepsilon}W 
$$
\item there exists a constant $C := C(\mu) >0$ depending on $\mu$ but independent of $\varepsilon >0$ and $W \in L^2[B_0] \cap [L^{q}(\mathbb R^2)]^2$ such that :
\begin{equation} \label{eq_controlfrac}
\|R_{\mu,\varepsilon} W\|_{L^2(\mathbb R^2)} \leq C \|W\|_{L^q(\mathbb R^2)}.
\end{equation}
\end{itemize}
\end{prop}

We postpone the proof of this proposition to Appendix \ref{app:fractional}.
For completeness, we provide a proof of {\bf Lemma \ref{lem:fractional}} with this proposition at-hand.

\begin{proof}[Proof of Lemma \ref{lem:fractional}]
The proof follows a standard regularization-compactness scheme.
Let $\mu \in (0,1/2)$ and  { $q \in (1,2)$ such that $\mu < 1/q-1/2$.} Given $W \in \mathcal L^2 \cap [L^q(\mathbb R^2)]^2$  and $\varepsilon \in (0,\infty)$ we can construct $(A+\varepsilon)^{-\mu} W.$ Formula \eqref{eq:sqrt2} with a Heinz-Kato argument imply then that 
$$
\| (A+\varepsilon)^{-\mu} W \|_{\mathcal L^2} \leq \|(\tilde{\mathcal A} + \varepsilon)^{-\mu} W\|_{L^2[B_0]}.
$$
However, we have that:
$$
(\tilde{\mathcal A} + \varepsilon)^{-\mu} W = (\tilde{\mathcal A}_0 + \varepsilon)^{-\mu} (\mathds {1}_{\mathcal F_0} W) + R_{\mu,\varepsilon} W
$$
For the first term, according to \cite[Eq. (2.2)]{KozonoOgawa-decay}  {(that holds componentwise in our setting)} and a Hardy-Littlewood-Sobolev inequality, there holds:
$$
\|(\tilde{\mathcal A}_0 + \varepsilon)^{-\mu} (\mathds {1}_{\mathcal F_0} W)\|_{L^2} \leq C \|W\|_{L^{q'}(\mathcal F_0)}
$$
where $1/q' = \mu + 1/2.$ We have then $q' \in (q,2)$ so that, by interpolation, we derive:
$$
\|(\tilde{\mathcal A}_0 + \varepsilon)^{-\mu} (\mathds {1}_{\mathcal F_0} W)\|_{L^2} \leq C ( \|W\|_{L^q(\mathcal F_0)} + \|W\|_{L^2(\mathcal F_0)})
$$
As for the other part, applying the previous proposition, we conclude that:
$$
\| R_{\mu,\varepsilon} W\|_{L^2} \leq C(\mu) \|W\|_{L^q}.
$$
Letting $\varepsilon \to 0,$ we have thus that $(A+\varepsilon)^{-\mu} W$ converges to some $V$ (in $\mathcal L^2$) that satisfies $A^{\mu}V = W$ with the expected control { $\|V\|_{\mathcal L^2} \leq C( \|W\|_{L^q} + \|W\|_{L^2(\mathcal F_0)})$}. 
\end{proof}


\section{Stability of the Oseen vortex} \label{sec:Cauchy}

In this section, we construct global-in-time solutions to \eqref{eq:init_prob1}-\eqref{eq:init_prob7} for arbitrary {$W_0 \in \mathcal L^2$} and analyze the long-time behavior for small perturbations of fully-developed Oseen vortex. 

\medskip

To this end, we have first the following useful estimates in the same spirit as Lemma 2.1 of \cite{Gallay_Maekawa__NS_obstacle} (so that we do not detail the proof):

\begin{lem} \label{lem:est_theta}
    \begin{enumerate}
    \item For any $p \in (2, \infty]$, there exists a constant $a_p > 0$ such that for all $t \geq 0$
    \begin{equation} \label{eq:est_theta1}
        \norm{\Theta (t)}_{L^p} \leq \frac{a_p}{(1 + t)^{\frac{1}{2} - \frac{1}{p}}}.
    \end{equation}

    \item For any $p \in (1, \infty]$, there exists $b_p > 0$ such that for all $t \geq 0$
    \begin{equation*}
        \norm{\nabla \Theta (t)}_{L^p} \leq \frac{b_p}{(1 + t)^{1 - \frac{1}{p}}}.
    \end{equation*}

    \item For all $t, s \geq 0$, we have
    \begin{equation} \label{eq:est_theta_l2}
        \norm{\Theta (t) - \Theta (s)}_{L^2}^2 \leq \frac{1}{4 \pi} \abs{\log \frac{1 + t}{1 + s}}.
    \end{equation}

    \item There exists a constant $\kappa_1 > 0$ such that for all $t, s \geq 0$,
    \begin{equation*}
        \norm{\nabla \Theta (t) - \nabla \Theta (s)}_{L^2}^2 \leq \kappa_1 \abs{\frac{1}{1+ t} - \frac{1}{1 + s}}.
    \end{equation*}
    \end{enumerate}
\end{lem}

We recall then that, contrary to \cite{Gallay_Maekawa__NS_obstacle}, we don't need to use a cut-off function. Indeed, the boundary conditions are here more suitable than the no-slip boundary condition of \cite{Gallay_Maekawa__NS_obstacle} for the Oseen vortex, since $\Theta$ is a pure rotation on $\partial B_0:$ $\Theta (t, x) = g (t,1) \, x^\perp$ on $\partial B_0$.
From this remark and the construction of the pressure $\Pi$ in the introduction, we obtain that, when plugging the ansatz \eqref{eq:ansatz} into
\eqref{eq:NS_mod1}-\eqref{eq:NS_mod6}, we may have a remainder term in the Newton laws only. Furthermore, we have the following proposition:

\begin{prop}
    For all $t \geq 0$, there exists $C > 0$ such that for all $t \geq 0$
    \begin{equation*}
        \abs{\int_{\partial B_0} x^\perp \cdot {\Sigma (\Theta (t),\Pi(t))} n \diff \sigma (x)} + \abs{\partial_t g (t, 1)} \leq C \frac{1}{(1 + t)^2}.
    \end{equation*}
    There also holds for all $t \geq 0$
    \begin{equation*}
        \int_{\partial B_0} { \Sigma (\Theta (t),\Pi(t))} n \diff \sigma (x) = 0.
    \end{equation*}
\end{prop}

In particular, we see that there is actually no remainder in the Newton law for the linear momentum. But there is one in the Newton law on the angular momentum:
\begin{equation*}
    \zeta (t) \coloneqq - \int_{\partial B_0} x^\perp \cdot{ \Sigma (\Theta(t),\Pi(t))} n \diff \sigma (x) - \mathcal{J} \partial_t g (t, 1).
\end{equation*}
The previous result yields the following estimate for this remainder.
\begin{cor} \label{cor:reminder_est}
    There exists $C > 0$ such that for all $t \geq 0$,
    \begin{equation*}
        \abs{\zeta (t)} \leq \frac{C}{(1 + t)^2}.
    \end{equation*}
\end{cor}
Eventually, going to capital-letter unknowns, we obtain with similar arguments 
as in \cite{tucsnaktakahashi2004} that we have a solution $(w,\ell_{w},\omega)$
to \eqref{eq:init_prob1}-\eqref{eq:init_prob7} if the associated $W$ satisfies
\eqref{eq:Duhamel} with 
\begin{multline*}
F_{\alpha}(s) =  { \alpha}\zeta(s) x^\perp \frac{\mathds{1}_{B_0}}{\mathcal{J}} -  { \alpha}\mathbb{P} \Bigl[ ((\Theta (s) \cdot \nabla) w (s) + (w (s) \cdot \nabla) \Theta (s) - (\ell_{W (s)} \cdot \nabla) \Theta (s)) \mathds{1}_{\mathcal F_0} \Bigr]  \\
      -  \mathbb P\Bigl[((w (s) - \ell_{W (s)}) \cdot \nabla) w (s)\mathds{1}_{\mathcal F_0}\Bigr] .
\end{multline*}


We proceed with the proof of {\bf Theorem \ref{thm_main1}} and {\bf Theorem \ref{thm:main2}}. We first study in the next subsection the Duhamel formula \eqref{eq:Duhamel}. { The analysis applies to the two cases.  Either we  start from a sufficiently large $t_0$ for small $\mathcal L^2$ data and we obtain {\bf Theorem \ref{thm:main2}}, or we do not restrict the size of initial data and obtain existence of a solution on a small time-interval.  This result is then complemented in the last subsection with an a priori  estimate to yield {\bf Theorem \ref{thm_main1}.}}

\subsection{Proof of {Theorem \ref{thm:main2}}}

The main result of this part is the following theorem:

\begin{theorem} \label{th:cauchy}
 \label{item:cauchy_th2} { Let $t_0 \geq 0$ and $W_0 \in \mathcal L^2.$ The two following items hold true:
\begin{itemize}
\item[i)] There exists $T>0$ such that the solution $W(t)$ of \eqref{eq:Duhamel} exists on $[t_0,t_0+T].$ Furthermore,  any upper bound on 
$|\alpha| + \|W_0\|_{\mathcal H^1}$ yields a lower bound on $T.$ 

\item[ii)]
 There exists positive constants $K_0$, $\delta_0$, $K_6$ and $T_0$ such that, if $t_0 \geq T_0$, if $\abs{\alpha} \leq \delta_0$, and if $\norm{W_0}_{\mathcal{L}^2} \leq K_6$, then the solution $W(t)$ of \eqref{eq:Duhamel} is global in time and satisfies
    \begin{equation*}
        \sup_{t \geq t_0} \norm{W (t)}_{\mathcal{L}^2} + \sup_{t > t_0} (t - t_0)^\frac{1}{2} (\norm{\nabla w (t)}_{L^2 (\mathcal{F}_0)} + \abs{\ell_{W (t)}}) \leq K_0 ( \norm{W_0}_{\mathcal{L}^2} + \abs{\alpha} (1 + t_0)^{- \frac{5}{4}} ).
    \end{equation*}
    In addition, if
    \begin{equation} \label{eq:assump_mu}
        M \coloneqq \sup_{\tau \geq 0} \tau^\mu \norm{S (\tau) W_0}_{\mathcal{L}^2} + \sup_{\tau > 0} \tau^{\mu + \frac{1}{2}} \Bigl( \norm{\nabla S (\tau) W_0}_{L^2 (\mathcal{F}_0)} + \abs{\ell_{S (\tau) W_0}} \Bigr) < \infty,
    \end{equation}
    for a fixed $\mu \in (0, \frac{1}{2})$, then
    \begin{equation*}
        \sup_{t \geq t_0} (t - t_0)^\mu \norm{W (t)}_{\mathcal{L}^2} + \sup_{t > t_0} (t - t_0)^{\mu + \frac{1}{2}} (\norm{\nabla w (t)}_{L^2 (\mathcal{F}_0)} + \abs{\ell_{W (t)}}) \leq 2M + C \abs{\alpha}
    \end{equation*}
    for some $C > 0$.
    \end{itemize}
    }
\end{theorem}

\begin{proof}
    The proof is very similar to the proof of Proposition 3.2 of \cite{Gallay_Maekawa__NS_obstacle}, who followed the classical { fixed-point} approach of Fujita and Kato \cite{Fujita_Kato__NS_initial_value_prob}.  
    { Below, we denote with $K$ a constant which depends only on the properties of the semi-group $S.$ This constant may vary between lines.}

    Given $t_0 \geq 0$ { and $T >0$, we introduce the Banach space $$
    X \coloneqq \mathcal{C}^0 ([t_0, t_0+T], \mathcal{L}^2) \cap \mathcal{C}^0 ((t_0, t_0+T), H^1 (\mathbb{R}^2) \cap L^\infty (B_0)),
    $$}equipped with the norm
    \begin{equation*}
        \norm{W}_X = \sup_{ t \in [t_0,t_0+T]} \norm{W (t)}_{\mathcal{L}^2} + { \dfrac{1}{\delta}} \sup_{ t \in  (t_0,t_0+T)} (t - t_0)^\frac{1}{2} (\norm{\nabla w (t)}_{L^2 (\mathcal{F}_0)} + \abs{\ell_{W (t)}}).
    \end{equation*}
    { with $\delta \in (0,1]$ to be fixed later on.}
    From Lemma \ref{lem:est_semigroup}, we know that { $W_H(t) := S(t-t_0) W_0$ satisfies:
    $$
    \|W_H(t)\|_{\mathcal L^2} \leq \|W_0\|_{\mathcal L^2}, \qquad 
   \sqrt{t-t_0} |\ell_{W_H(t)}| \leq  K \min(1,\sqrt{T}) \|W_0\|_{\mathcal L^2}\,,
   \qquad 
   \sqrt{t-t_0}\|\nabla w_H(t)\|_{L^2(\mathcal F_0 )} \leq K \|W_0\|_{\mathcal L^2}.
    $$
{Then, if $W_0 \in \mathcal H^1$, we have
\begin{equation*}
    \norm{\nabla w_H(t)}_{L^2(\mathcal F_0 )} \leq K \|W_0\|_{\mathcal H^1}.
\end{equation*}
On the other hand, if $W_0 \in \mathcal L^2$,} by classical regularizing arguments,  we can write $W_0 = W_0^{\varepsilon} + (W_0-W_{0}^{\varepsilon})$ where $W_0^{\varepsilon} \in \mathcal H^1$ and $(W_0- W_0^{\varepsilon})$ is arbitrary small in $\mathcal L^2.$ From the properties of $S,$ we have then that, for arbitrary $\varepsilon >0$ there exists $C_{\varepsilon}$ such that: 
    $$
    \sqrt{t-t_0} \norm{\nabla w_H(t)}_{L^2(\mathcal F_0 )} \leq   C_{\varepsilon}\sqrt{T} + K \varepsilon. 
    $$

{If we take $K \varepsilon = \delta^2$, we obtain that there exists a constant $K_{\delta}$ depending on $\delta$ for which:}
\begin{equation} \label{eq_boundWH}
\|W_H(t)\|_X \leq K \left ( 1 + \dfrac{ \min (1,\sqrt{T})}{\delta} \right) \|W_0\|_{\mathcal L^2} + \min \left(K \dfrac{ \|W_0\|_{\mathcal L^2}}{\delta} ,  K_{\delta} \sqrt{T}  + {\delta}, \frac{K \sqrt{T}}{\delta} \|W_0\|_{\mathcal H^1} \right).
\end{equation}
}
 
Then, given any $W \in X$, we denote for $t \geq t_0$:
    \begin{gather*}
        F_0 (t) = \int_{t_0}^t S(t-s) \mathbb{P} \Bigl[ \zeta(s) x^\perp \frac{\mathds{1}_{B_0}}{\mathcal{J}} \Bigr] \diff s, \\
        (F_1 W) (t) = \int_{t_0}^t S(t-s) \mathbb{P} \Bigl[\left(  (\Theta (s) \cdot \nabla) w (s) \right)\Bigr] \diff s, \\
        (F_2 W) (t) = \int_{t_0}^t S(t-s) \mathbb{P} \Bigl[ (w (s) - \ell_{W(s)})  \cdot \nabla  w (s) \Bigr] \diff s, \\
        (F_3 W) (t) = \int_{t_0}^t S(t-s) \mathbb P \Bigl[ (w(s) - \ell_{W (s)}) \cdot \nabla  \Theta (s) \Bigr] \diff s, \\
        (F W) (t) = \alpha F_0 (t) + \alpha (F_1 W) (t) + (F_2 W) (t) + \alpha (F_3 W) (t).
    \end{gather*}

    We show that $F$ maps $X$ into $X$ and that:
    \begin{align}
        \norm{F W}_X & \leq K \left( { \dfrac{|\alpha|}{\delta} \min\left( T^{3/4}  ,  \frac{1}{(1 + t_0)^\frac{5}{4}}\right)}  + \abs{\alpha} 
        { \left( \sqrt{\delta} + \dfrac{1}{\delta} \min (1, T^{\frac 14})\right)}\norm{W}_X \right)\notag \\
        & \quad + K { \min (1, \sqrt{\delta}+T^{\frac 12})}\norm{W}_X^2 , \label{eq:est_FW} \\
        \norm{F W_1 - F W_2}_X &  \leq K \left( \abs{\alpha}   { \left( \sqrt{\delta} + \dfrac{1}{\delta} \min (1, T^{\frac 14})\right)} +{ \min (1, \sqrt{\delta}+T^{\frac 12})}\left(  \norm{W_1}_X + \norm{W_2}_X\right)\right)  \dots \notag \\
        & \dots \norm{W_1 - W_2}_X. \label{eq:FW_lip}
    \end{align}
{    For this, we compute now  bounds successively for $F_0,$ $F_1,$ $F_2,$ and $F_3$.}
    First, using Corollary \ref{cor:reminder_est} and Lemma \ref{lem:est_semigroup} (with $q = \frac{4}{3}$), we get for all $t \geq t_0$
    \begin{align*}
        \norm{F_0 (t)}_{\mathcal{L}^2} + { \dfrac{(t - t_0)^\frac{1}{2}}{\delta}} (\norm{\nabla F_0 (t)}_{L^2 (\mathcal{F}_0)} + \norm{F_0 (t)}_{L^\infty}) & \leq C \int_{t_0}^t \Biggl( \frac{1}{(t-s)^\frac{1}{4}} + { \frac{(t - t_0)^\frac{1}{2}}{\delta(t - s)^\frac{3}{4}}} \Biggr) \frac{1}{(1 + s)^2} \diff s \\
        &          \leq { \dfrac{1}{\delta} \min\left( T^{3/4}  ,  \frac{1}{(1 + t_0)^\frac{5}{4}}\right).}
    \end{align*}
Then, we control $F_2$ with the help of Lemma \ref{lem:bilinear}  (see Appendix \ref{app:tech}) 
 which ensures that:
     \begin{equation*}
        \norm{F_2 (t)}_{\mathcal{L}^2} + { \dfrac{(t - t_0)^\frac{1}{2}}{\delta}} (\norm{\nabla F_2 (t)}_{L^2 (\mathcal{F}_0)} + \norm{F_2 (t)}_{L^\infty}) \leq {K\min (1, \sqrt{\delta}+T^{\frac 12})} \|W\|_{X}^2
    \end{equation*}    
Similarly, there holds
    \begin{align*}
        \norm{(w - \ell_{W (s)} \cdot \nabla) \Theta (s)}_{L^\frac{4}{3} (\mathcal{F}_0)} &  \leq
        \|w\|_{L^4(\mathcal F_0)} \|\nabla \Theta (s)\|_{L^2(\mathcal F_0)}
        +   
         \abs{\ell_{W (s)}} \norm{\nabla \Theta (s)}_{L^\frac{4}{3} (\mathcal{F}_0)}  \\
         & \leq C \left(\dfrac{{ \sqrt{\delta}}}{(s-t_0)^{\frac 14}} \dfrac{1}{(1+s)^{\frac 12}} + \frac{{\delta}}{(s - t_0)^\frac{1}{2}} \frac{1}{(1+s)^\frac{1}{4}}\right) \norm{W}_X
    \end{align*}
so that, applying the boundedness of $\mathbb P: L^{4/3}(\mathbb R^2) \to \mathcal L^{4/3}$ (see \cite[Remark 2.4]{WangXin2011}):
    \begin{multline*}
        \norm{F_3 W (t)}_{\mathcal{L}^2} + {\dfrac{(t - t_0)^\frac{1}{2}}{\delta}} (\norm{\nabla F_3 W (t)}_{L^2 (\mathcal{F}_0)} + \norm{F_3 W (t)}_{L^\infty}) \\
            \begin{aligned}
                &\leq C \int_{t_0}^t \Biggl( \frac{1}{(t-s)^\frac{1}{4}} + \frac{(t - t_0)^\frac{1}{2}}{{ \delta}(t - s)^\frac{3}{4}} \Biggr) \left( \dfrac{{ \sqrt{\delta}}}{(s-t_0)^{\frac 14}} \dfrac{1}{(1+s)^{\frac 12}} + \frac{{ {\delta}}}{(s - t_0)^\frac{1}{2}} \frac{1}{(1 + s)^\frac{1}{4}}\right)  \diff s \, \norm{W}_X \\
                &\leq C { \left( \sqrt{\delta} + \dfrac{1}{\sqrt{\delta}} \min\left( 1,  T^{\frac 14} \right)  \right) }\norm{W}_X.
            \end{aligned}
    \end{multline*}
We finally bound $F_1 W.$ To this end, the procedure is similar to that of \cite{Gallay_Maekawa__NS_obstacle}. First, we observe that $\Theta \cdot n = 0$ on $\partial B$
so that we can rewrite {(see {\bf Corollary \ref{cor_zecor}})}: 
    \begin{equation*}
      { S(\tau) \mathbb P[\mathbf{1}_{\mathcal F_0} (\Theta \cdot \nabla) w ] = A^\frac{1}{2} S(\tau) A^{- \frac{1}{2}} \mathbb{P}{\rm div}( \mathbf{1}_{\mathcal F_0}\Theta \otimes w ) \,, \qquad \forall \, \tau >0. }
    \end{equation*}
    Moreover, using Lemma \ref{lem:A_div_bounded} and the estimate \eqref{eq:est_theta1}, we compute
    \begin{equation*}
        \norm{A^{- \frac{1}{2}} \mathbb{P} \divg (\mathds{1}_{\mathcal F_0} \Theta \otimes w )}_{L^2 (\mathcal{F}_0)} \leq  \norm{\Theta (s) w (s)}_{L^2 (\mathcal{F}_0)} \leq \dfrac{C}{(1 + s)^{\frac{1}{2}}} \norm{W}_X.
    \end{equation*}
    The above remark with Lemma \ref{lem:est_semigroup} and \eqref{eq:link_A_nabla}  lead to:
    \begin{align*}
        \norm{(F_1 W) (t)}_{\mathcal{L}^2} &\leq \int_{t_0}^t (t-s)^{- \frac{1}{2}} \norm{A^{- \frac{1}{2}} \mathbb{P} \divg(  \mathds{1}_{\mathcal F_0} \Theta \otimes w  (s))}_{\mathcal{L}^2} \diff s \\
            &\leq C \int_{t_0}^t (t-s)^{- \frac{1}{2}} (1 + s)^{-\frac{1}{2}} \norm{W}_X \diff s \\
            & { \leq C \min(1, T^{\frac 12})\norm{W}_X},
    \end{align*}
    and
    \begin{align*}
      { \dfrac{ (t - t_0)^\frac{1}{2} }{\delta} }(\norm{\nabla (F_1 W) (t)}_{L^2 (\mathcal{F}_0)} + \abs{\ell_{F_3W (t)}}) &\leq
            \begin{multlined}[t][10cm]
          { \dfrac{1}{\delta}}      \int_{t_0}^{\frac{t + t_0}{2}} \frac{(t - t_0)^\frac{1}{2}}{t-s} \norm{A^{- \frac{1}{2}} \mathbb{P} \divg (  \mathds{1}_{\mathcal F_0} \Theta \otimes w ) (s)}_{\mathcal{L}^2} \diff s \\ 
                +   { \dfrac{1}{\delta}}   \int_{\frac{t + t_0}{2}}^t \frac{(t - t_0)^\frac{1}{2}}{(t-s)^\frac{1}{2}} \norm{(\Theta (s) \cdot \nabla) w (s) }_{L^2 (\mathcal{F}_0)} \diff s
            \end{multlined} \\
            &\leq 
            \begin{multlined}[t][10cm]
                C \biggl[    { \dfrac{1}{\delta}}   \int_{t_0}^{\frac{t + t_0}{2}} \frac{(t - t_0)^\frac{1}{2}}{t-s} (1 + s)^{-\frac{1}{2}} \diff s \, \norm{W}_X \\
                +    { \dfrac{1}{\delta}}   \int_{\frac{t + t_0}{2}}^t \frac{(t - t_0)^\frac{1}{2}}{(t-s)^\frac{1}{2}} \frac{1}{(1 + s)^\frac{1}{2} (s - t_0)^\frac{1}{2}} \diff s \, \norm{W}_X \biggr]
            \end{multlined} \\
            &\leq C{  \dfrac{1}{\delta} \min \left( 1 ,  T^{\frac 12}\right)} \norm{W}_X.
    \end{align*}
    Since $F W = \alpha F_0 + \alpha F_1 W + F_2 W + \alpha F_3 W$, this concludes the proof of \eqref{eq:est_FW}. The Lipschitz bound \eqref{eq:FW_lip} is established in the same way {from the fact that $F_1$ and $F_3$ are linear in $W$}.

\medskip    
       
{ We proceed with the proof of item $ii)$. For this, we fix $\delta=1$ so that \eqref{eq:est_FW} and \eqref{eq:FW_lip} entail:
    \begin{align*}
   &  \norm{W_H}_X \leq 2K \|W_0\|_{\mathcal L^2}\,, \\ 
    &    \norm{F W}_X \leq K \left(   \frac{|\alpha|}{(1 + t_0)^\frac{5}{4}} + \abs{\alpha}  \norm{W}_X + \norm{W}_X^2 \right),  \\   
   &     \norm{F W_1 - F W_2}_X \leq K \left( \abs{\alpha} + \norm{W_1}_X + \norm{W_2}_X \right) \norm{W_1 - W_2}_X. 
    \end{align*}
} 
Let $T = \infty$ and $r > 0$ such that $4 K r \leq 1$ and define $B_r \coloneqq \{ W \in X \, | \norm{W}_X \leq r \}$. If we assume that $4 \abs{\alpha} K \leq 1$, $4 K \norm{W_0}_{\mathcal{L}^2} \leq r$ and $4 K \abs{\alpha} (1 + t_0)^{-\frac{5}{4}} \leq r$, then { the previous estimates} imply that the map $W \mapsto S(t-t_0) W_0 + F W$ leaves the closed ball $B_r$ invariant and is a strict contraction in $B_r$. By construction, the unique fixed point of this map in $B_r$ is the desired solution of \eqref{eq:Duhamel}. This proves the existence part of Theorem \ref{th:cauchy} with
    \begin{equation*}
        K_0 = K, \qquad
       { \delta_0} = \frac{1}{4 K}, \qquad
        K_6 = \frac{1}{16 \, K^2}, \qquad
        T_0 = (4 K)^\frac{4}{5}.
    \end{equation*}

    In a second step, we assume that \eqref{eq:assump_mu} holds for some $\mu \in (0, \frac{1}{2})$. Given any $T > t_0$, we denote
    \begin{equation*}
        \mathcal{E}_T = \sup_{t_0 \leq t \leq T} (t - t_0)^\mu \norm{W (t)}_{\mathcal{L}^2} + \sup_{t_0 < t \leq T} (t - t_0)^{\mu + \frac{1}{2}} (\norm{\nabla w (t)}_{L^2 (\mathcal{F}_0)} + \abs{\ell_{W (t)}}),
    \end{equation*}
    where $W$ (also represented by the triplet $(w, \ell_W, \omega_W)$) is the solution of \eqref{eq:Duhamel} previously constructed. Since $W(t) = S(t-t_0) W_0 + (F W) (t)$, we have
    \begin{equation*}
        \mathcal{E}_T \leq M + \sup_{t_0 \leq t \leq T} (t - t_0)^\mu \norm{(F W) (t)}_{\mathcal{L}^2} + \sup_{t_0 < t \leq T} (t - t_0)^{\mu + \frac{1}{2}} (\norm{\nabla (F W) (t)}_{L^2 (\mathcal{F}_0)} + \abs{\ell_{(F W) (t)}}),
    \end{equation*}
    where $M$ is defined in \eqref{eq:assump_mu}. Let $p \in (1, 2)$ be such that $\frac{1}{p} > \mu + \frac{1}{2}$ and define $q \in (2, \infty)$ such that $\frac{1}{q} = \frac{1}{p} - \frac{1}{2}$. In particular, $\frac{1}{2} > \frac{1}{q} > \mu$. First, we have in a similar way as previously:
    \begin{multline*}
        (t - t_0)^\mu \norm{F_0 (t)}_{\mathcal{L}^2} + (t - t_0)^{\mu + \frac{1}{2}} (\norm{\nabla F_0 (t)}_{L^2 (\mathcal{F}_0)} + \norm{F_0 (t)}_{L^\infty}) \\
            \begin{aligned}
                &\leq C \int_{t_0}^t \Biggl( \frac{(t - t_0)^\mu}{(t-s)^\frac{1}{q}} + \frac{(t - t_0)^{\mu + \frac{1}{2}}}{(t - s)^\frac{1}{p}} \Biggr) \frac{1}{(1 + s)^2} \diff s \\
                &\leq \frac{C}{(1 + t_0)^{\frac{1}{q} - \mu + 1}}.
            \end{aligned}
    \end{multline*}
    %
%
%
%
%

    The same computations as previously can be done for $F_1 W,$ $F_2 W$ and $F_3 W$
    introducing the further decay of $W$ induced by $\mathcal E_T$ (see \cite{Gallay_Maekawa__NS_obstacle} for more details), so that we finally get
    \begin{equation} \label{eq:last_est}
        \mathcal{E}_T \leq M + \tilde{K} (\abs{\alpha} (1 + t_0)^{\mu - \frac{1}{q} - 1} + \abs{\alpha} \mathcal{E}_T + \norm{W}_X \mathcal{E}_T),
    \end{equation}
    for some positive constant $\tilde{K}$ independent of $T$ and $t_0$. Taking $\delta$ and $K_6$ smaller and $t_0$ larger if needed, we can ensure that $2 \tilde{K} ( \abs{\alpha} + \norm{W}_X ) \leq 1$, so that \eqref{eq:last_est} leads to
    \begin{equation*}
        \mathcal{E}_T \leq 2 M + 2 \tilde{K} \abs{\alpha} (1 + t_0)^{\mu - \frac{1}{q} - 1},
    \end{equation*}
    for all $T > t_0$. 
    
  \medskip
      
 { We finally prove item $i)$ similarly.  For this we remark that,  with the same computations as above,  we can take first $r = 4 K \|W_0\|_{\mathcal L^2}$ and choose $\delta$ small enough and then $T$ small enough (depending on $|\alpha|,W_0$ and $r$) so that  given $W,W_1,W_2 \in B_r$ inequalities yield
  \eqref{eq_boundWH}  \eqref{eq:est_FW} and \eqref{eq:FW_lip} entail:
     \begin{align*}
   &    \|W_H(t)\|_X \leq 2 K \|W_0\|_{\mathcal L^2} \\
     &   \norm{F W}_X \leq K  \|W_0\|_{\mathcal L^2} +\frac{1}{2} \norm{W}_X,  \\   
       & \norm{F W_1 - F W_2}_X \leq \dfrac{1}{2} \norm{W_1 - W_2}_X. 
    \end{align*}
We conclude similarly as above this yields a unique fixed point.  In case 
$W_0 \in \mathcal H^1,$ we essentially add the further remark that, looking at   \eqref{eq_boundWH}  \eqref{eq:est_FW} and \eqref{eq:FW_lip},  we can choose $r,\delta,T$ depending only on $|\alpha|,\|W_0\|_{\mathcal H^1}$ to reach the same inequalities. 
  }  
\end{proof}

{
To conclude this part,  we point out that, for the linearized system we have the decay estimates of Lemma \ref{lem:est_semigroup}. }Hence we infer the content of 
{\bf Theorem \ref{thm:main2}} by remarking that, if $W_{0} \in \mathcal L^2 \cap \mathcal L^q$ (meaning that $w_0 \in L^q(\mathcal F_0)$) is small in $\mathcal{L}^2$ then the assumption \eqref{eq:assump_mu} is satisfied with $\mu = \frac 1q - \frac 12.$

\subsection{A logarithmic energy estimate}

In this section we complement the proof of {\bf Theorem \ref{thm_main1}}. 
This is the content of the next lemma:

\begin{lem}
    There exists a constant $K > 0$ such that, for any $\alpha \in \mathbb{R}$ and any $W_0 \in \mathcal{L}^2$, the solution of \eqref{eq:Duhamel} with initial data $W_0$ provided by Theorem \ref{th:cauchy} is global in time and satisfies, for all $t \geq 0$,
    \begin{equation*}
        \norm{W (t)}_{\mathcal{L}^2}^2 + \int_0^t \norm{D (w(s))}_{L^2 (\mathcal{F}_0)} \diff s \leq K \Bigl( \norm{W_0}_{\mathcal{L}^2}^2 + \alpha^2 \log{(1 + t)} + K_\alpha \Bigr),
    \end{equation*}
    where $K_\alpha = \alpha^2 (1 + \log{(1 + \abs{\alpha})})$.
\end{lem}

\begin{proof}
    Fix $\alpha \in \mathbb{R}$, $W_0 \in \mathcal{L}^2$, and let $W \in \mathcal{C}^0 ([0, T], \mathcal{L}^2) \cap \mathcal{C}^0 ((0, T], \mathcal{H}^1)$ be the solution provided by Theorem \ref{th:cauchy} with initial data $W(0) = W_0$.  { We recall that we denote $V = W + \alpha \Theta.$ Given any $\tau \geq 0$,  we define then,}
    \begin{gather*}
        \tilde{w}_\tau (t, x) = v (t, x) - \alpha \Theta (t + \tau, x) = w(t, x) + \alpha \Bigl( \Theta (t,x) - \Theta (t + \tau, x) \Bigr), \qquad \textnormal{for all } x \in \mathcal{F}_0, \\
        \ell_{\tilde{W}_\tau (t)} = \ell_{V} (t) = \ell_{W (t)}, \\
        \omega_{\tilde{W}_\tau (t)} = \omega_V - \alpha g(t+\tau,1) = \omega_{W (t)} + \alpha \Bigl( g(t, 1) - g(t+\tau, 1) \Bigr).
    \end{gather*}
    The given $\tilde{W}_\tau$ (represented by the triplet $(\tilde{w}_\tau, \ell_{\tilde{W}_\tau (t)}, \omega_{\tilde{W}_\tau (t)})$) satisfy the system of equations \eqref{eq:init_prob1}-\eqref{eq:init_prob7} (or equivalently \eqref{eq:Duhamel}), where $\Theta (t)$ and $\zeta (t)$ are replaced by $\Theta (t+\tau)$ and $\zeta (t+\tau)$. Assume first that the solutions are smooth enough. Multiplying both sides of \eqref{eq:init_prob1} by $\tilde{w}_\tau$ and integrating by parts over $\mathcal{F}_0$ (using the fact that $\tilde{w}_\tau$ and $\Theta$ are divergence-free), we find
    \begin{multline*}
        \frac{1}{2} \frac{\diff}{\diff t} \norm{\tilde{w}_\tau}_{L^2 (\mathcal{F}_0)}^2 + 2 \norm{D(\tilde{w}_\tau (t))}_{L^2 (\mathcal{F}_0)}^2 = \int_{\partial B_0} \tilde{w}_\tau (t) \cdot \Sigma (\tilde{w}_\tau (t)) n \diff \sigma (x) - \frac{\alpha}{2} \int_{\partial B_0} \Theta (t+\tau) \abs{\tilde{w}_\tau}^2 \cdot n \diff \sigma (x) \\
            - \alpha \int_{\mathcal{F}_0} \tilde{w}_\tau (t) \cdot \Bigl( (\tilde{w}_\tau (t) - \ell_{\tilde{W}_\tau (t)} \Bigr) \cdot \nabla \Bigr) \Theta (t+\tau) \diff x + \frac{1}{2} \int_{\partial B_0} \abs{\tilde{w}_\tau (t)}^2 (\tilde{w}_\tau (t) - {\ell_{\tilde{W}_\tau (t)}}) \cdot n \diff \sigma (x).
    \end{multline*}
    Since $\Theta$ and ${\tilde{w}_\tau (t) - \ell_{\tilde{W}_\tau (t)}}$ are orthogonal to $n$ on $\partial B_0$, the second and fourth terms also vanish. \eqref{eq:init_prob3}-\eqref{eq:init_prob5} then yield
    \begin{multline*}
        \frac{1}{2} \frac{\diff}{\diff t} \Bigl( \norm{\tilde{w}_\tau (t)}_{L^2 (\mathcal{F}_0)}^2 + m \abs{\ell_{\tilde{W}_\tau (t)}}^2 + \mathcal{J} \abs{\omega_{\tilde{W}_\tau (t)}}^2 \Bigr) + 2 \norm{D(\tilde{w}_\tau (t))}_{L^2 (\mathcal{F}_0)}^2 \\
            = - \alpha \int_{\mathcal{F}_0} \tilde{w}_\tau (t) \cdot \Bigl( (\tilde{w}_\tau (t) - \ell_{\tilde{W}_\tau (t)} \Bigr) \cdot \nabla \Bigr) \Theta (t+\tau) \diff x + \alpha \zeta (t+\tau) \, \omega_{\tilde{W}_\tau (t)}.
    \end{multline*}
    The right-hand side can be estimated as usual with Lemma \ref{lem:est_theta}:
    \begin{gather*}
        \abs{\int_{\mathcal{F}_0} \tilde{w}_\tau (t) \cdot \Bigl( \tilde{w}_\tau (t) \cdot \nabla \Bigr) \Theta (t+\tau) \diff x} \leq \norm{\tilde{w}_\tau (t)}_{L^2 (\mathcal{F}_0)}^2 \frac{C}{1 + \tau + t}, \\
        \abs{\int_{\mathcal{F}_0} \tilde{w}_\tau (t) \cdot \Bigl( \ell_{\tilde{W}_\tau (t)} \cdot \nabla \Bigr) \Theta (t+\tau) \diff x} \leq \norm{\tilde{w}_\tau (t)}_{L^2 (\mathcal{F}_0)} \abs{\ell_{\tilde{W}_\tau (t)}} \frac{C}{(1 + \tau + t)^\frac{1}{2}}, \\
        \abs{\zeta (t+\tau) \, \omega_{\tilde{W}_\tau (t)}} \leq \frac{C}{(1 + \tau + t)^2} \abs{\omega_{\tilde{W}_\tau (t)}} \leq \frac{C}{(1 + \tau + t)^2} \Bigl( \abs{\omega_{\tilde{W}_\tau (t)}}^2 + 1 \Bigr)
    \end{gather*}
    Integrating in time from $0$ to $t$ for any $t > 0$ leads to
    \begin{multline} \label{eq:interm_est_log}
        \frac{1}{2} \norm{\tilde{W}_\tau (t)}_{\mathcal{L}^2}^2 + 2 \int_0^t \norm{D(\tilde{w}_\tau (s))}_{L^2 (\mathcal{F}_0)}^2 \diff s \\
            \leq \frac{1}{2} \norm{\tilde{W}_\tau (0)}_{\mathcal{L}^2}^2 + K \abs{\alpha} \int_0^t \biggl( \frac{\norm{\tilde{w}_\tau (s)}_{L^2 (\mathcal{F}_0)}^2}{1 + \tau + s} + \frac{\norm{\tilde{w}_\tau (s)}_{L^2 (\mathcal{F}_0)} \abs{\ell_{\tilde{W}_\tau (s)}}}{(1 + \tau + s)^\frac{1}{2}} + \frac{\abs{\omega_{\tilde{W}_\tau (s)}}^2}{(1 + \tau + s)^2} + \frac{1}{(1 + \tau + s)^2} \biggr) \diff s,
    \end{multline}
    for some constant $K > 0$, independent of $\tau$ in particular. Such an estimate then also holds for weaker solutions.
    From this estimate, for $\tau = 0$, the Gronwall lemma shows that $\norm{\tilde{W}_\tau (t)}_{\mathcal{L}^2}^2$ is bounded locally in time. 
    { Adapting for instance \cite[pp. 69-70]{tucsnaktakahashi2004},  we infer that $\|W_{\tau}(t)\|_{\mathcal H^1}$ does not blow in finite time either. Therefore,  { item $i)$} of Theorem \ref{th:cauchy} yields that our solution  $W$ is global in time.}
    Then, for general $\tau \geq 0$, we need to better estimate the second term, in particular $\abs{\ell_{\tilde{W}_\tau (s)}}$ which should decrease faster than $\norm{\tilde{w}_\tau (s)}_{L^2 (\mathcal{F}_0)}$ (or $\norm{\tilde{W}_\tau (s)}_{\mathcal{L}^2}$ equivalently). For this, we use Corollary \ref{cor:est_l_GN}. Applying it for $p = 2 + \log{(1 + \tau + s)}$, we get: 
    \begin{align*}
        \abs{\ell_{\tilde{W}_\tau (s)}} 
            &\leq C (2 + \log{(1 + \tau + s)})^\frac{1}{2} \norm{\tilde{W}_\tau (s)}_{L^2 (\mathbb{R}^2)}^{\frac{2}{2 + \log{(1 + \tau + s)}}} \norm{\nabla \tilde{W}_\tau (s)}_{L^2 (\mathbb{R}^2)}^{1-\frac{2}{2 + \log{(1 + \tau + s)}}} \\
            &\leq C (2 + \log{(1 + \tau + s)})^\frac{1}{2} \norm{\tilde{W}_\tau (s)}_{\mathcal{L}^2}^{\frac{2}{2 + \log{(1 + \tau + s)}}} \norm{D (\tilde{W}_\tau (s))}_{L^2 (\mathcal{F}_0)}^{1-\frac{2}{2 + \log{(1 + \tau + s)}}},
    \end{align*}
    where we have used Lemma \ref{lem:grad_sym} in the last estimate. Then, we obtain:
    \begin{align*}
        \frac{\norm{\tilde{w}_\tau (s)}_{L^2 (\mathcal{F}_0)} \abs{\ell_{\tilde{W}_\tau (s)}}}{(1 + \tau + s)^\frac{1}{2}} &\leq C \biggl(\frac{2 + \log{(1 + \tau + s)}}{1 + \tau + s}\biggr)^\frac{1}{2} \norm{\tilde{W}_\tau (s)}_{\mathcal{L}^2}^{1 + \frac{2}{2 + \log{(1 + \tau + s)}}} \norm{D (\tilde{W}_\tau (s))}_{L^2 (\mathcal{F}_0)}^{1-\frac{2}{2 + \log{(1 + \tau + s)}}} \\
            &\leq \norm{D (\tilde{W}_\tau (s))}_{L^2 (\mathcal{F}_0)}^2 + C \biggl[ \frac{2 + \log{(1 + \tau + s)}}{1 + \tau + s} \biggr]^{\xi (\tau + s)} \norm{\tilde{W}_\tau (s)}_{\mathcal{L}^2}^2,
    \end{align*}
    where
    \begin{equation*}
        \xi (x) \coloneqq \frac{1}{1 + \frac{2}{2 + \log{(1+x)}}} = 1 - \frac{1}{2 + \log{(1+x)}} \frac{2}{1 + \frac{2}{2 + \log{(1+x)}}}.
    \end{equation*}
    In particular, we can easily compute that
    \begin{equation*}
        \biggl[ \frac{2 + \log{(1 + \tau + s)}}{1 + \tau + s} \biggr]^{\xi (\tau + s)} \leq C \frac{2 + \log{(1 + \tau + s)}}{1 + \tau + s}.
    \end{equation*}
    Therefore, we obtain
    \begin{multline*}
        \frac{1}{2} \norm{\tilde{W}_\tau (t)}_{\mathcal{L}^2}^2 + \int_0^t \norm{D(\tilde{w}_\tau (s))}_{L^2 (\mathcal{F}_0)}^2 \diff s \\
            \leq \frac{1}{2} \norm{\tilde{W}_\tau (0)}_{\mathcal{L}^2}^2 + \frac{K \abs{\alpha}}{1+\tau} + K \abs{\alpha} \int_0^t \norm{\tilde{W}_\tau (s)}_{L^2 (\mathcal{F}_0)}^2 \frac{2 + \log{(1 + \tau + s)}}{1 + \tau + s} \diff s,
    \end{multline*}
    By applying the Gronwall lemma, we get
    \begin{multline*}
        \frac{1}{2} \norm{\tilde{W}_\tau (t)}_{\mathcal{L}^2}^2 + \int_0^t \norm{D(\tilde{w}_\tau (s))}_{L^2 (\mathcal{F}_0)}^2 \diff s \\
        \leq K \Bigl[ \norm{\tilde{W}_\tau (0)}_{\mathcal{L}^2}^2 + \frac{\abs{\alpha}}{1+\tau} \Bigr] \exp{\Bigl[K \abs{\alpha} ( \log{(1 + \tau + t)}^2 - \log{(1 + \tau)}^2 + \log{(1 + \tau + t)} - \log{(1 + \tau)} )\Bigr]}
    \end{multline*}
    Now take $\tau = (\chi t)^2$ where $\chi = 1 + \abs{\alpha}$, we get:
    \begin{equation*}
        \log{(1 + \tau + t)} - \log{(1 + \tau)} = \log{\Bigr( 1 + \frac{t}{1+ (\chi t)^2}\Bigl )} \leq C
    \end{equation*}
    and
    \begin{align*}
        \log{(1 + \tau + t)}^2 - \log{(1 + \tau)}^2 &= \log{\Bigr( 1 + \frac{t}{1+ (\chi t)^2}\Bigl )} \Bigl( \log{(1 + t + (\chi t)^2)} + \log{(1 + (\chi t)^2)} \Bigr) \\
            &\leq C \frac{t \log{(1 + (\chi t)^2)}}{1 + (\chi t)^2} \leq \frac{C}{\chi}.
    \end{align*}
    Thanks to the estimate \eqref{eq:est_theta_l2} and the explicit expression of $g(t,r)$, there also holds
    \begin{align*}
        \norm{\tilde{W}_\tau (0)}_{\mathcal{L}^2}^2 &\leq 2 \norm{W_0}_{\mathcal{L}^2} + 2 \alpha^2 \norm{\Theta (0) - \Theta (\tau)}_{L^2 (\mathbb{R}^2)} + 2 \alpha^2 \abs{g(0,1) - g(\tau, 1)} \\
            &\leq \norm{W_0}_{\mathcal{L}^2}^2 + C \alpha^2 \Bigl( 1 + \log{(1+(\chi t)^2)} \Bigr) \\
            &\leq \norm{W_0}_{\mathcal{L}^2}^2 + C \alpha^2 \Bigl( 1 + \log{(1 + \abs{\alpha})} + \log{(1 + t)} \Bigr),
    \end{align*}
    but also
    \begin{align*}
        \norm{W (t)}_{\mathcal{L}^2}^2 &\leq 2 \norm{\tilde{W}_\tau (t)}_{\mathcal{L}^2}^2 + 2 \alpha^2 \norm{\Theta (t+\tau) - \Theta (t)}_{\mathcal{L}^2}^2 \\
            &\leq 2 \norm{\tilde{W}_\tau (t)}_{\mathcal{L}^2}^2 + \frac{\alpha^2}{2 \pi} \log{\Bigl( 1 + \frac{t}{1 + (\chi t)^2} \Bigr)} \\
            &\leq 2 \norm{\tilde{W}_\tau (t)}_{\mathcal{L}^2}^2 + \frac{\alpha^2}{2 \pi} \frac{t}{1 + (\chi t)^2} \\
            &\leq 2 \norm{\tilde{W}_\tau (t)}_{\mathcal{L}^2}^2 + \frac{\alpha^2}{2 \chi \pi},
    \end{align*}
    and
    \begin{align*}
        \int_0^t \norm{D(w (s))}_{L^2 (\mathcal{F}_0)}^2 \diff s &\leq 2 \int_0^t \norm{D(\tilde{w}_\tau (s))}_{L^2 (\mathcal{F}_0)}^2 \diff s + 2 \alpha^2 \int_0^t \norm{D(\Theta (\tau + s) - \Theta (s))}_{L^2 (\mathcal{F}_0)}^2 \diff s \\
            &\leq 2 \int_0^t \norm{D(\tilde{w}_\tau (s))}_{L^2 (\mathcal{F}_0)}^2 \diff s + 2 \alpha^2 \int_0^t \norm{\nabla (\Theta (\tau + s) - \Theta (s))}_{L^2 (\mathcal{F}_0)}^2 \diff s \\
            &\leq 2 \int_0^t \norm{D(\tilde{w}_\tau (s))}_{L^2 (\mathcal{F}_0)}^2 \diff s + 2 \kappa_1 \alpha^2 \int_0^t \Bigl( \frac{1}{1 + \tau} - \frac{1}{1 + \tau + s} \Bigr) \diff s \\
            &\leq 2 \int_0^t \norm{D(\tilde{w}_\tau (s))}_{L^2 (\mathcal{F}_0)}^2 \diff s + 2 \kappa_1 \alpha^2 \Bigl( \frac{t}{1 + (\chi t)^2} - \log{(1 + \frac{t}{1 + (\chi t)^2})} \Bigr) \diff s \\
            &\leq 2 \int_0^t \norm{D(\tilde{w}_\tau (s))}_{L^2 (\mathcal{F}_0)}^2 \diff s + 2 \kappa_1 \frac{\alpha^2}{\chi}.
    \end{align*}
    The last five estimates put together (along with $\chi \geq 1$) lead to the result.
\end{proof}

\section{Global stability for finite-energy solutions}\label{sec:bounded}
This last section is devoted to the proof of {\bf Theorem \ref{thm_main3}}. For this, we first recall the partial result in \cite{Ervedoza_Hillairet__disk_fluid} on which relies our analysis:

\begin{lem}[{\cite[Theorem 1.3]{Ervedoza_Hillairet__disk_fluid}}]
Let $q \in (1,2)$ and assume that $V_0 \in \mathcal L^q \cap \mathcal L^2$ with $\|V_0\|_{\mathcal L^2}$ sufficiently small.
Then the unique finite-energy weak solution $V$ with initial data $V_0$ satisfies:
\begin{align}
& \sup_{t >0} t^{\frac 1p - \frac 1q} \|V(t)\|_{\mathcal L^p}  <\infty && \quad \forall \, p \in (2,\infty) \label{eq:decayV} \\
& \sup_{t >0} t^{\frac 1q} |\ell_{v}(t)| < \infty. \label{eq:decayl}
\end{align}
\end{lem}

{\bf Theorem  \ref{thm_main3}} is then a direct consequence of the two following propositions
that we prove in the next subsections:
\begin{prop} \label{prop:Lq}
Let $q \in (1,2)$ and asume that  $V_0 \in \mathcal L^q \cap \mathcal L^2$. Then the unique finite-energy solution $V$ starting from $V_0$
satisfies:
\begin{align}
& V \in C([0,\infty) ; \mathcal L^q \cap \mathcal L^2) \\
& \nabla V \in  L^1_{loc}([0,\infty); L^q(\mathcal F_0) \cap L^2(\mathcal F_0) ) \\
 & \ell_{V} \in L^2((0,\infty)).
\end{align}
\end{prop}

\begin{prop} \label{prop:decay}
Let $q \in (1,2)$ and asume that  $V_0 \in \mathcal L^q \cap \mathcal L^2$. Then the unique finite-energy solution $V$ starting from $V_0$
satisfies:
\begin{equation}
\liminf_{t \to \infty} \|V(t)\|_{\mathcal L^2} = 0.
\end{equation}
\end{prop}

\subsection{Proof of Proposition \ref{prop:Lq}}
Let $q<2$ and $V_0 \in \mathcal L^q \cap \mathcal L^2$. We recall that, by the construction of \cite{tucsnaktakahashi2004}, we have
$V \in C([0,\infty); \mathcal L^2)$ and  $\nabla V \in L^2((0,\infty); L^2(\mathbb R^2)).$ Furthermore, with the proof of {\bf Theorem \ref{th:cauchy}} we know that the solution $V$ is computed through the Duhamel formula:
\begin{equation} \label{eq:Duhamel2}
V(t) = S(t) V_0 + \int_0^t S(t-s) \mathbb P[ \mathbf{1}_{\mathcal F_0} (V-\ell_V) \cdot \nabla V ]{\rm d}s.
\end{equation}
since it is the only fixed point of the mapping:
$$
\mathcal D: W \mapsto S(t) V_0 + \int_0^t S(t-s) \mathbb P[\mathbf 1_{\mathcal F_0} (W -\ell_{W}) \cdot \nabla W]{\rm d}s.
$$
in the space $C([0,T]; \mathcal L^2) \cap C((0,T) ; H^1(\mathcal F_0))$ endowed with the $X$-norm:
$$
\|W\|_{X} = \sup_{[0,T]} \|W\|_{\mathcal L^2} + \sup_{[0,T]} { \sqrt{t} \|\nabla W\|_{L^2(\mathcal F_0)}}
$$ 
(for $T$ sufficiently small). We show here that 
the same property holds adding the property $V \in C([0,T]; \mathcal L^q) \cap C((0,T) ; W^{1,q}(\mathcal F_0)).$ 
Let fix $B_T$ the subset in 
$C([0,T]; \mathcal L^2 \cap \mathcal L^q) \cap C((0,T) ;  H^1(\mathcal F_0) \cap W^{1,q}(\mathcal F_0 ))$ 
containing $W$ satisfying 
$$
 \|W\|_{X} \leq  2 \|V_0\|_{\mathcal L^2}, \quad  \|W\|_{X_q} := \sup_{[0,T]} \left( \|W(t)\|_{\mathcal L^q}  + \sqrt{t} \|\nabla W(t)\|_{L^q(\mathbb R^2)}\right) \leq (1+ K_2) ( \|V_0\|_{\mathcal L^q} + \|V_0\|_{\mathcal L^2}),
$$
where $K_2$ is the constant involved in {\bf  Lemma \ref{lem:est_semigroup}}.
By adapting the computations in the proof of {\bf Theorem \ref{th:cauchy}}, we obtain a time $T_0$ sufficiently small such that 
for $T <T_0$ the above mapping is a contraction on $B_T$ for the $X$-norm. Then, given $W \in B_T,$ applying the duality estimates in 
Lemma \ref{lem:duality_est} with $p=q$ we obtain that 
$$
 \|\mathcal D[W](t)\|_{\mathcal L^q} \leq \|V_0\|_{\mathcal L^q} + \int_0^t \phi_q(t-s) \|(W-\ell_{W} )\otimes W\|_{L^q(\mathcal F_0)} {\rm ds} \\
\quad \forall \, t \in [0,T]$$
where 
$$
\phi_q(s) = K_5 
\begin{cases}
 s^{-1/2} & \text{ if $s <1$} \\
s^{-1+\frac 1q} & \text{if $s > 1$}
\end{cases} 
$$
The last integral we denote $I[W]$ is then bounded by applying the Gagliardo Nirenberg inequality:
\begin{align*}
I[W] & \leq \int_{0}^t |\phi_q(t-s)| \left( \|W\|_{\mathcal L^2}^{\frac 1q} \|\nabla W\|^{2 (1- \frac 1q)}_{L^2(\mathcal F_0)} \right) + |\ell_W| \|W\|_{\mathcal L^q} {\rm d}s \\
& \leq \int_0^t |\phi_q(t-s)| \left( s^{-(1-1/q)}  \|W\|^2_{X} + \|W\|_{X} \|W\|_{\mathcal L^q}\right) \diff s.
\end{align*}
At this point, we realize that, for $T < 1$ there is an absolute constant $\tilde{K}_5$ for which:
$$
\sup_{[0,T]} \int_0^t \phi_q(t-s) s^{-(1-1/q)} \leq \tilde{K}_5 T^{\frac 1q - \frac 12} \qquad  \sup_{[0,T]}\int_0^t \phi_q(t-s) \leq \tilde{K}_5 \sqrt{T}.
$$
Since $q < 1/2$ we can take $T_0$ smaller  (but  decreasingly in the quantity $\|V_0\|_{\mathcal L^2} + \|V_0\|_{\mathcal L^q}$) so that for $T <T_0$:
\begin{equation} \label{eq_T1}
\sup_{[0,T]} \|\mathcal D[W](t)\|_{\mathcal L^q}  \leq \|V_0\|_{\mathcal L^q} + \tilde{K}_5 T^{\frac 1q - \frac 12}\|V_0\|_{\mathcal L^2} \left( \|V_0\|_{\mathcal L^2} + \sup_{[0,T_1]}\|W\|_{\mathcal L^q} \right) \leq  \|V_0\|_{\mathcal L^q} + \|V_0\|_{\mathcal L^2}.
\end{equation}
As for the gradient, we apply semi-group estimates of  {Lemma \ref{lem:est_semigroup}} to yield that
$$
\sqrt{t} \|\nabla \mathcal D[W](t)\|_{L^q(\mathcal F_0)} \leq K_2 \|V_0\|_{\mathcal L^q} + K_2 \int_0^t \left(\dfrac{t}{t-s} \right)^{\frac 12} \|(W-\ell_w) \cdot \nabla W\|_{L^q(\mathcal F_0)} {\rm d}s. 
$$
Combining then H\"older inequalities (where $1/q^* = 1/q-1/2$) together with a Gagliardo-Nirenberg inequality (interpolating the $L^{q^*}$-norm between the $L^2$ and $\dot{H}^1$ norms) and the {bound already obtained} on $\|W\|_X$, we bound:
\begin{align*}
& \|W\cdot \nabla W\|_{L^q(\mathcal F_0)} \leq \|W\|_{L^{q^*}(\mathcal F_0)} \|\nabla W\|_{L^2(\mathcal F_0)} \leq s^{\frac 32 - \frac 1q}\|W\|_{X}^2
\\
&  \|\ell_W \cdot \nabla W\|_{L^q(\mathcal F_0)} \leq \sqrt{s} \|W\|_{X} \|W\|_{X_q}  
\end{align*}
Since $s < 1,$ we end up with:
$$
\sqrt{t} \|\nabla \mathcal D[W](t)\|_{L^q(\mathcal F_0)} \leq K_2 \|V_0\|_{\mathcal L^q} + K_2 \int_0^t \left(\dfrac{t\ s}{t-s} \right)^{\frac 12}  {\rm d}s\|V_0\|_{\mathcal L^2} (\|V_0\|_{\mathcal L^2} + \|V_0\|_{\mathcal L^q}).
$$
By a homogeneity argument we have:
$$
\int_0^t \left(\dfrac{t\ s}{t-s} \right)^{\frac 12} \leq C t^{\frac 32},
$$
hence we can choose $T_0$ smaller if necessary (but decreasing in the quantity $\|V_0\|_{\mathcal L^2} + \|V_0\|_{\mathcal L^q}$)  so that:
for $T <T_0$:
$$
\sup_{[0,T]} \sqrt{t} \|\nabla \mathcal D[W](t)\|_{L^q(\mathcal F_0)} \leq K_2 (\|V_0\|_{\mathcal L^q} + \|V_0\|_{\mathcal L^2}).
$$
Finally, $\mathcal D$ maps $B$ into $B.$ With similar computations, we obtain that it is a contraction up to restrict to a smaller $T_0$
again and conclude that we propagate the property $V \in \mathcal L^q$ and $\nabla V \in L^q(\mathcal F_0)$ on a short time-interval. We note that on this time-interval { $\Delta T$}, we have 
\begin{equation} \label{eq_bound_W1q}
\|\nabla W\|_{L^1(0,\Delta T) ; L^2(\mathcal F_0)} \leq \|W\|_X \qquad \|\nabla W\|_{L^1(0,\Delta T) ; L^q(\mathcal F_0)} \leq \|W\|_{X_q} 
\end{equation}

To obtain further that $V \in \mathcal L^q$ and $\nabla V \in L^q(\mathcal F_0)$ for all times we remark that by a standard blow-up alternative, it is sufficient to obtain local 
bounds for $\|V(t)\|_{\mathcal L^q} + \|V(t)\|_{\mathcal L^2}.$ Since this is already known for $\|V(t)\|_{\mathcal L^2}$ we focus here on 
$\|V(t)\|_{\mathcal L^q}.$ To this end, we note that choosing $T_1$ so that $\tilde{K}_5 T_1^{\frac 1q - \frac 12} \|V_0\|_{\mathcal L^2} \leq 1/2$ and applying $\eqref{eq_T1}$ with $V$ we have
$$
\sup_{[0,T_1]} \|V(t)\|_{\mathcal L^q} \leq 2 \|V_0\|_{\mathcal L^q} + \|V_0\|_{\mathcal L^2}.
$$
Furthermore, since our system of equation is autonomous, we can reproduce this computation starting from any $t_0 >0.$ Finally, since 
we already have a uniform bounds for $\|V(t_0)\|_{\mathcal L^2}$ we obtain that there exists a short time increment $T_1$ (independent of the initial data) so that for arbitrary $t_0 >0:$
$$
\sup_{[t_0,t_0+T_1]} \|V(t)\|_{\mathcal L^q} \leq 2 \|V(t_0)\|_{\mathcal L^q} + \|V_0\|_{\mathcal L^2}.
$$
In particular, there can be no blow-up of $\|V(t)\|_{\mathcal L^q}$ in finite-time. Then on a time-interval $[0,T]$ since we have an {\em a priori} bound for $\|V\|_{\mathcal L^q} + \| V \|_{\mathcal L^2},$ we can see our solution as a concatenation of local-in-time solutions
constructed as above on a small-time interval $\Delta T.$ By concatenating the remarks \eqref{eq_bound_W1q}
on the time-intervals $[n\Delta T,(n+1)\Delta T]$ we conclude that 
$$
\nabla V \in L^1((0,T) ; L^2(\mathcal F_0) \cap L^q(\mathcal F_0)).
$$

\medskip

To complete the proof of {\bf Proposition \ref{prop:Lq}}, we show now that $\ell_{V} \in L^2([0,\infty)).$ Since $\nabla V \in L^2([0,\infty)),$ we first remark that:
\begin{equation} \label{eq_Teps}
\forall \, \varepsilon >0,  \quad \exists\, T_{\varepsilon} >0 \text{ s.t. } \int_{T_{\varepsilon}}^{\infty} \|\nabla V\|_{L^2}^2 < \varepsilon.
\end{equation}
Thanks to the representation formula \eqref{eq:Duhamel2}, we have then that, for arbitrary $t >0$ we can split 
$\ell_{v}(t) = \ell_S(t) + \ell_{NL}(t)$ where:
$$
\ell_S(t) = \ell_{S(t)V_0} \quad \ell_{NL}(t) = \ell_{I(t)} 
$$
where
$$
I(t) = \int_0^t S(t-s) \mathbb P[\mathbf 1_{\mathcal F_0} (V -\ell_{V}) \cdot \nabla V]{\rm d}s.
$$
Since $V_0 \in \mathcal L^2 \cap \mathcal L^q$ we apply {\bf Lemma \ref{lem:est_semigroup}} to yield that:
$$
|\ell_S(t)| \leq \min \left( 1, \dfrac{1}{t^{\frac 1q}}\right) \|V_0\|_{\mathcal L^q} \in L^2((0,\infty)).
$$
For the nonlinear term, we apply the duality estimates of {\bf Lemma \ref{lem:est_semigroup}} with $r > 2$. We obtain:
$$
|\ell_{NL}(t)| \leq \int_0^{t} \dfrac{1}{(t-s)^{\frac 12 + { \frac 1r}}} \|(V-\ell_v) \otimes V \|_{L^r(\mathcal F_0)} 
$$
At this point, let fix $T >0$ (sufficiently large) and remark that the right-hand side can be seen as a truncated (time-)convolution of
$1/s^{\frac 12 + \frac 1r}$ and $\|(V-\ell_v) \otimes V \|_{L^r(\mathcal F_0)} \mathbf{1}_{[0,T]}.$ 
By a Hardy-Littlewood-Sobolev inequality, we have then:
\begin{align*}
\|\ell_{NL}\|_{L^2(0,T)} & \leq  \| \ |\cdot|^{-(\frac 12 + \frac 1r)} * \|(V-\ell_v) \otimes V \|_{L^r(\mathcal F_0)}\|_{L^2{(0,T)}} \\
& \leq  C_{r} \left( \int_0^T \|(V-\ell_v) \otimes V\|^p_{L^r(\mathcal F_0)} \right)^{\frac 1p}
\leq C_r  \left( \int_0^T \|V\|^{2p}_{L^{2r}(\mathcal F_0)} \right)^{\frac 1p} + C_{r} \left( \int_0^T \||\ell_v| V\|^p_{L^r(\mathcal F_0)} \right)^{\frac 1p}
\end{align*}
where $p$ is the conjugate exponent of $r.$ For the first-integral on the right-hand side, we apply again a Gagliardo-Nirenberg inequality
and the fact that $p$ is the conjugate exponent of $r$ to yield that:
$$
\int_0^T \|V\|^{2p}_{L^{2r}(\mathcal F_0)}  \leq C_r \sup_{[0,T]}\|V(t)\|^{\frac{2p}{r}}_{\mathcal L^2} \int_0^{T} \|\nabla V\|^2_{L^2} \leq C_r \|V_0\|^{2p}_{\mathcal L^2}
$$
To estimate the last term, we introduce an intermediate time $T_{mid}$ to be fixed later on. We note here that, for arbitrary $0 \leq T_1< T_2$  combining a standard
H\"older inequality and a Gagliardo Nirenberg inequality entails that (since $p<2$):
\begin{align*}
\int_{T_1}^{T_2} \||\ell_v| V\|^p_{L^r(\mathcal F_0)} 
& \leq  \left( \int_{T_1}^{T_2} |\ell_v|^2\right)^{\frac p2} \left( \int_{T_1}^{T_2} \|V\|^{\frac{2p}{2-p}}_{L^{r}}\right)^{1 - \frac p2}\\
& \leq \|\ell_{v}\|^p_{L^2(T_1,T_2)} \sup_{(T_1,T_2)} \|V\|^{\frac{2p}{r}}_{L^2}
\left( \int_{T_1}^{T_2} \|\nabla V\|^{\frac{2p}{2-p}(1-\frac{2}{r})}_{L^2}\right)^{ 1- \frac p2}
\end{align*}
Recalling that $p$ and $r$ are conjugate exponents yield that 
$$
\frac{2p}{2-p}(1-\frac{2}{r}) = 2
$$
and we infer that:
\begin{equation*}
\int_{T_1}^{T_2} \||\ell_v| V\|^p_{L^r(\mathcal F_0)} 
\leq C \|\ell_{v}\|^p_{L^2(T_1,T_2)} \|V_0\|^{\frac{2p}{r}}_{\mathcal L^2}
\left( \int_{T_1}^{T_2} \|\nabla V\|^{2}_{L^2} \right)^{1- \frac p2 }
\end{equation*}
When $T > T_{mid}$, combining the previous computations  between $T_1=0$ and $T_2 = T_{mid}$ and between $T_{1} = T_{mid}$ and $T_2 = T,$ we conclude that:
\begin{multline*}
\|\ell_{NL}\|_{L^2(0,T)} \leq C_r \|V_0\|_{\mathcal L^2}^2 
+ C_r\|\ell_v\|_{L^2(0,T_{mid})} \|V_0\|^{\frac 2r}_{\mathcal L^2} \left( \int_0^{T_{mid}} \|\nabla V\|^2_{L^2} \right)^{\frac 1p - \frac 12}
\\
+  C_r\|\ell_v\|_{L^2(T_{mid},T)} \|V_0\|^{\frac 2r}_{\mathcal L^2} \left( \int_{T_{mid}}^{T} \|\nabla V\|^2_{L^2} \right)^{\frac 1p - \frac 12}.
\end{multline*}
At this point, we recall the remark \eqref{eq_Teps} and choose $T_{mid}$ so that:
$$
C_r\|V_0\|^{\frac{2p}{r}}_{\mathcal L^2} \left( \int_{T_{mid}}^{\infty} \|\nabla V\|^2_{L^2} \right)^{\frac 1p - \frac 12} < \frac 12
$$
Splitting ${\ell_v} = \ell_{S} + \ell_{NL}$ and arguing that, on compact time-interval, we can always control  $|\ell_{v}|$ by 
$\|V\|_{\mathcal L^2} \leq \|V_0\|_{\mathcal L^2},$ we infer that :
$$
\|\ell_v\|_{L^2(0,T)} \leq  \|\ell_S\|_{L^2(0,\infty)} + C_r (1+ \sqrt{T_{mid}})\|V_0\|^2_{\mathcal L^2} +  \frac{1}{2} \|\ell_v\|_{L^2(0,T)}.
$$
Eventually, we conclude that, for arbitrary $T >T_{mid}$ we have:
$$
\|\ell_v\|_{L^2(0,T)} \leq  C_{q,r} \left( \|V_0\|_{\mathcal L^q} + (1+ \sqrt{T_{mid}})\|V_0\|^2_{\mathcal L^2}\right).
$$
This concludes the proof.

\subsection{Proof of Proposition \ref{prop:decay}}
This proof is inspired of \cite[Section 5]{Gallay_Maekawa__NS_obstacle}.
Let $q<2$ and $V \in \mathcal L^2 \cap \mathcal L^q$. We recall that we take $\mu < \frac 1q- \frac 12$ so that 
$\mathcal L^2 \cap \mathcal L^q \subset D(A^{-\mu})$. 
Thanks to {\bf Proposition \ref{prop:Lq}}, we have that the unique finite-energy solution satisfies
\begin{itemize}
\item $V \in C([0,\infty); \mathcal L^q \cap \mathcal L^2) \cap C((0,\infty) ; W^{1,q}(\mathbb R^2) \cap H^1(\mathbb R^2)).$
\item $\mathbf 1_{\mathcal F_0}(V-\ell_v) \cdot \nabla V \in L^1_{loc}((0,\infty) ; L^2(\mathbb R^2) \cap L^q(\mathbb R^2))$
\end{itemize}
In particular, we have:
$$
\partial_t V - AV = \mathbb P ((V - \ell_v) \cdot \nabla V) \qquad V_{|_{t=0}} = V_0
$$
where $V \in C([0,\infty) ; D(A^{-\mu}))$ and $ \mathbb P ((V - \ell_v) \cdot \nabla V) \in L^1_{loc}(0,\infty ; D(A^{-\mu})).$
Consequently, we can apply the operator $A^{-\mu}$ to this equation and we obtain that $U = A^{-\mu} V$ is a mild solution to:
$$
\partial_t U - AU = A^{-\mu} \mathbb P ((V-\ell_v) \cdot \nabla V) \qquad U_{|_{t=0}} = U_0.
$$
We have in particular for arbitrary $t >0$ that
$$
\dfrac{1}{2} \|U(t)\|^2_{\mathcal L^2} + \int_0^{t} \|\nabla U(s)\|^{2}_{L^2} {\rm d}s \leq \int_0^{t} \langle A^{-\mu} \mathbb P((V-\ell_v) \cdot \nabla V) , U \rangle {\rm d}s.
$$
However, for arbitrary $s \in (0,t)$ there holds:
\begin{align*}
\left| \langle A^{-\mu} \mathbb P((V-\ell_v) \cdot \nabla V) , U \rangle\right| &= \left| \int_{\mathcal F_0} ([(V-\ell_v) \cdot \nabla] A^{-\mu} U) \cdot V \diff x \right| \\
& \leq  \Bigl(\|V\|^2_{L^4(\mathcal F_0)} + \|V\|_{L^2} |\ell_v| \Bigr) \|A^{\frac 12-\mu} U\|_{\mathcal L^2} \\
& \leq C\Bigl(|\ell_v| + \|\nabla V\|_{L^2(\mathcal F_0)} \Bigr) \|A^{\mu} U\|_{\mathcal L^2}\|A^{\frac 12-\mu} U\|_{\mathcal L^2},
\end{align*}
where we applied a Gagliardo-Nirenberg inequality to pass from the second to the last line. 
At this point, we argue by interpolation that
\begin{equation*}
    \|A^{\mu} U\|_{\mathcal L^2}\|A^{\frac 12-\mu} U\|_{\mathcal L^2} \leq C \|U\|_{\mathcal L^2}\|A^{\frac 12} U\|_{\mathcal L^2},
\end{equation*}
thus
\begin{equation*}
    C\Bigl(|\ell_v| + \|\nabla V\|_{L^2(\mathcal F_0)} \Bigr) \|U\|_{\mathcal L^2}\|A^{\frac 12} U\|_{\mathcal L^2} \leq C\Bigl(|\ell_v| + \|\nabla V\|_{L^2(\mathcal F_0)} \Bigr)^2 \|U\|_{\mathcal L^2}^2 + \frac{1}{2} \|A^{\frac 12} U\|_{\mathcal L^2}^2.
\end{equation*}
This yields finally that, for all $t \geq 0$:
$$
\|U(t)\|^2_{\mathcal L^2} + \int_0^{t} \|\nabla U(s)\|^{2}_{L^2} {\rm d}s \leq C \int_0^{t} \Bigl(|\ell_v| + \|\nabla V\|_{L^2(\mathcal F_0)}\Bigr)^2\|U (s)\|^2_{\mathcal L^2} {\rm d}s.
$$ 
Eventually a Gronwall lemma yields that:
$$
\|U(t)\|^2_{\mathcal L^2} + \int_0^{t} \|\nabla U(s)\|^{2}_{L^2} {\rm d}s \leq \|U_0\|^2_{\mathcal L^2} \exp\Bigl[ \int_0^t C\Bigl(|\ell_v| + \|\nabla V\|_{L^2(\mathcal F_0)}\Bigr)^2 \Bigr]
$$
Since the integral in the exponential is bounded by Proposition \ref{prop:Lq}, we have then a uniform bound 
$$
\sup_{t >0} \|U(t)\|^2_{\mathcal L^2} + \int_0^{\infty} \|\nabla U(s)\|^{2}_{L^2} {\rm d}s \leq C_0
$$
where the constant $C_0$ depends on the whole solution $V$ (and {\em a priori} not only on $V_0$).

At this point, we argue in the same manner as in \cite[Corollary 4.2]{Gallay_Maekawa__NS_obstacle}. The situation is even more favorable
since we have uniform bounds. Indeed, since $\nabla U \in L^2((0,\infty) ; L^2(\mathbb R^2)),$ we can construct a sequence of times
$t_n$ growing to infinity such that $\|\nabla U(t_n)\|_{L^2} \to 0.$ We have then that $\|A^{1/2} U(t_n)\|_{\mathcal L^2}$ goes to $0$
while $\|U(t_n)\|_{\mathcal L^2}$ remains bounded. By interpolation, $\|V(t_n)\|_{\mathcal L^2} = \|A^{\mu} U(t_n)\|_{\mathcal L^2}$ (where $\mu < 1/2$)  goes also to $0$ as $n$ goes to infinity. This ends the proof.

\appendix
\section{Technical lemmas} \label{app:tech}

We gather in this section technical lemmas used throughout the paper.
We start with handling nonlinearities in the Duhamel formula. { We recall that,
given $t_0 >0$ { and $T >0$ we denote: 
$$
X := C([t_0,t_0+T] ; \mathcal L^2 ) \cap C((t_0,t_0+T] ; H^1(\mathbb R^2) \cap L^{\infty}(B_0))
$$
that we endow with the norm:
$$
\|W\|_X := \sup_{t \geq t_0} \|W(t)\|_{\mathcal L^2} + \sup_{t > t_0} \dfrac{(t-t_0)^{\frac 12}}{\delta} \left( \|\nabla w(t)\|_{L^2(\mathcal F_0)} + |\ell_{W}(t)| \right).
$$
}
Other notations are introduced in Section \ref{sec:Cauchy}.

\begin{lem} \label{lem:bilinear}
Let {$t_0,T> 0.$ }Given  $(W_a,W_b) \in X$ we denote:
$$
F(t) : =\int_{t_0}^t S(t-s) \mathbb P[ \mathbf{1}_{\mathcal F_0}(w_a - \ell_a)\cdot \nabla w_b]{\rm d}s
 \quad \forall \, t \in {(t_0,t_0+T).}
$$
Then there holds:
\begin{itemize}
\item $F \in X$
\item there exists a constant $C >0$ for which:
$$
\|F\|_X \leq C \, { \min (1, \sqrt{\delta}+ T^{\frac 12})} { \|W_a\|_{X} \|W_b\|_X}
$$
\end{itemize}
\end{lem} 
We emphasize that, in this lemma, the assumption ${ W_a} \in X$ induces that, for every $s {\in (t_0,t_0+T)},$
{$W_a(s)$} is a rigid motion on $B_0$. Obviously, we denote $\ell_a$ the translation velocity (with respect to the origin) associated with this motion.   

\begin{proof}
We only give a proof of the second item. To this end, we remark that, since $w_a$ is divergence free: 
$$
(w_a - \ell_a) \cdot \nabla w_b = {\rm div} ((w_a - \ell_a) \otimes w_b)\,, \quad \text{on $\mathcal F_0$}. 
$$
Since $(w_a - \ell_a) \cdot n = 0$ on $\partial B$ we can then extend by $0$ to create an $L^2(\mathbb R^d)$-source term which fulfills the assumptions of \cite[Corollary 3.10]{Ervedoza_Hillairet__disk_fluid}.
%
%
This yields, for arbitrary { $t \in [t_0,t_0+T]$ (recalling also that $\delta <1$)} 
\begin{align*}
\|F(t)\|_{\mathcal L^2} & \leq K\int_{t_0}^t \dfrac{1}{\sqrt{t-s}} \|(w_a - \ell_a) \otimes w_b \|_{L^2(\mathcal F_0)} \\
& \leq K \int_{t_0}^t \dfrac{1}{\sqrt{t-s}} \left( |\ell_a| \|W_b\|_{\mathcal L^2} + \|W_a\|^{1/2}_{\mathcal L^2} \|\nabla w_a\|^{1/2}_{L^2(\mathcal F_0)} \|W_b\|^{1/2}_{\mathcal L^2} \|\nabla w_b\|^{1/2} _{L^2(\mathcal F_0)}\right) \\
& \leq K \int_{t_0}^{t}  \dfrac{1}{\sqrt{t-s}} \dfrac{{ \delta}}{\sqrt{s-t_0}} {\rm d}s \|W_a\|_{X} \|W_b\|_X.
\end{align*}
Hence, we have $\|F(t)\|_{\mathcal L^2} \leq C{ \delta} \|W_a\|_{X} \|W_b\|_X.$

\medskip

{ For the second part, we first split $F = F_1 + F_2 + F_3$ with $t_{mid} = (t+t_0)/2$ and denote:}
\begin{align*}
F_1(t) & = \int_{t_0}^{t_{mid}} S(t-s) \mathbb P {\rm div} [ \mathbf{1}_{\mathcal F_0} (w_a - \ell_a) \otimes w_b ] {\rm d}s \\
F_2(t) &=    \int_{t_{mid}}^{t} S(t-s) \mathbb P[w_a \cdot \nabla w_b]  {\rm d}s\\
F_3(t)  &=  \int_{t_{mid}}^{t} S(t-s) \mathbb P[\ell_{a} \cdot \nabla w_b]{\rm d}s.
\end{align*}
For the first term, we combine \cite[Corollary 3.10]{Ervedoza_Hillairet__disk_fluid} with standard continuity properties of $S.$ Remarking that:
$$
F_1(t) =  S(t-t_{mid}) \int_{t_0}^{t_{mid}} S(t_{mid}-s) \mathbb P {\rm div} [ \mathbf{1}_{\mathcal F_0} (w_a - \ell_a) \otimes w_b ] {\rm d}s,
$$
we obtain with obvious notations and similar computations that:
\begin{align*}
|\ell_{1}(t)| + \|\nabla F_1(t)\|_{L^2(\mathcal F_0)} & \leq \dfrac{1}{\sqrt{t-t_{mid}}} \norm{\int_{t_0}^{t_{mid}} S(t_{mid}-s) \mathbb P {\rm div} [ \mathbf{1}_{\mathcal F_0} (w_a - \ell_a) \otimes w_b ] {\rm d}s}_{\mathcal L^2} \\
& \leq \dfrac{C}{\sqrt{t-t_0}} \int_{t_0}^{t_{mid}} \dfrac{1}{\sqrt{t_{mid}-s}} \|(w_a - \ell_a) \otimes w_b \|_{L^2(\mathcal F_0)} \\
&  \leq \dfrac{C{ \delta}}{\sqrt{t-t_0}}  \|W_a\|_{X} \|W_b\|_{X}.
\end{align*}
For the other terms, we apply standard continuity properties of $S.$ First we note that 
$w_{a} \cdot \nabla w_b \in L^{4/3}(\mathbb R^2)$ with:
\begin{align*}
\|w_{a} \cdot \nabla w_b\|_{L^{4/3}(\mathbb R^2)} & \leq \|w_a\|^{\frac 12}_{\mathcal L^2}\|\nabla w_a\|_{L^2(\mathbb R^2)}^{\frac 12} \|\nabla w_{b}\|_{L^2(\mathbb R^2)}  \\
& \leq \dfrac{{ \delta^{\frac 32}}}{(t-t_0)^{3/4}} \|w_a\|_{X} \|w_b\|_{X}
\end{align*}
Since $\mathbb P : L^{4/3}(\mathbb R^2) \to \mathcal L^{4/3}$ is bounded (see \cite[Remark 2.4]{WangXin2011}), we infer that:
\begin{align*}
|\ell_{2}(t)| + \|\nabla F_2(t)\|_{L^2(\mathcal F_0)} & \leq  K { \delta^{\frac 32}}
\int_{t_{mid}}^{t} \dfrac{1}{(t-s)^{3/4}} \dfrac{1}{(t-t_0)^{3/4}} {\rm d}s \|w_a\|_{X} \|w_b\|_{X} \\
& \leq \dfrac{C{ \delta^{\frac 32}}}{\sqrt{t-t_0}} \|W_a\|_{X} \|W_b\|_{X}.
\end{align*}
Finally, we bound (applying the standard continuity of $\mathbb P : L^2(\mathbb R^2) \to \mathcal L^2 $)
\begin{align*}
|\ell_{3}(t)| + \|\nabla F_3(t)\|_{L^2(\mathcal F_0)} & \leq  K 
\int_{t_{mid}}^{t} \dfrac{1}{\sqrt{t-s}} |\ell_a| \|\nabla w_b\|_{L^2(\mathbb R^2)} {\rm d}s \\
& \leq K \int_{t_{mid}}^{t} \dfrac{{ \delta^2}}{\sqrt{t-s}} \dfrac{{\rm d}s}{s-t_0} \|W_a\|_{X} \|W_b\|_{X} \\
& \leq \dfrac{C{ \delta^2}}{\sqrt{t-t_0}} \|W_a\|_{X} \|W_b\|_{X}.  
\end{align*}

{ Finally,  we split again $F = F_1 + F_2 + F_3$ with $t_{mid} =t_0$ so that $F_1 =0$.  We remark then that similar estimate holds for $(F_2,\ell_2)$ while, for $F_3,$ we note that we can bound $|\ell_a| \leq \|W_a\|_{\mathcal L^2}$ to obtain:
\begin{align*}
|\ell_{3}(t)| + \|\nabla F_3(t)\|_{L^2(\mathcal F_0)} & \leq  K 
\int_{t_0}^{t} \dfrac{1}{\sqrt{t-s}} |\ell_a|  \|\nabla w_b\|_{L^2(\mathbb R^2)} {\rm d}s \\
& \leq K \int_{t_0}^{t} \dfrac{{ \delta}}{\sqrt{t-s}} \dfrac{{\rm d}s}{\sqrt{s-t_0}}  \|W_a\|_{X} \|W_b\|_{X} \\
& \leq C \delta \|W_a\|_{X} \|W_b\|_{X}.  
\end{align*}
This concludes the proof.
}

\end{proof}  

\section{Proof of Proposition \ref{prop:fractional}} \label{app:fractional}

We provide here a proof of Proposition \ref{prop:fractional}. 
To estimate $(\tilde{\mathcal A} +\varepsilon)^{-\mu}$ as required in the content of {Proposition \ref{prop:fractional}} we rely on the integral representation (because $\tilde{\mathcal A}$ is a positive selfadjoint operator, see \cite[Section 2.6]{Pazy}):
\begin{equation} \label{eq_repfract}
(\tilde{\mathcal A} + \varepsilon)^{-\mu} = \dfrac{\sin(\pi \mu)}{\pi}\int_{0}^{\infty} \dfrac{1}{(\lambda+\varepsilon)^{\mu}} (\tilde{\mathcal A} + \lambda+\varepsilon)^{-1}{\rm d}\lambda .
\end{equation}
In order to construct $R_{\mu,\varepsilon}$ we work at first on  a construction of $(\tilde{\mathcal A} + \lambda)^{-1}$ involving $(\tilde{\mathcal A}_0 + \lambda)^{-1}$ for $\lambda >0.$ 
To this end, we introduce objects that are crucial to the analysis. 

\medskip

We recall here basics on some modified Bessel functions. 
The following statements are taken from \cite[Section 8]{Olver}.
The function $K_0: (0,\infty) \to \mathbb R$ is the unique smooth solution to:
$$
- \dfrac{1}{r}\frac{\textrm{d}}{{\rm d}r}\left[ r \frac{\textrm{d}}{{\rm d}r} K_0(r)\right]  + K_0(r) = 0 \quad \forall \, r >0,
$$
that behaves asymptotically like:
$$
K_0(r) \sim 
\{
\begin{aligned}
& \sqrt{\dfrac{\pi}{2r}} \exp(-r) & \text{when $r \to \infty$} \\
& - \ln(r) & \text{ when $r \to 0$.}
\end{aligned}
\right. 
$$
Furthermore, all derivatives of $K_0$ enjoy the same decay at infinity as $K_0$ and $K_0'(r) \sim -1/r$ in $0$.
We mention also that $K_0 \geq 0$ and $K'_0 \leq 0$ on $(0,\infty)$ (see \cite[Theorem 8.1]{Olver}).
Similarly, $K_1:(0,\infty) \to \mathbb R$ is the smooth solution to:
$$
- \dfrac{1}{r}\frac{\textrm{d}}{{\rm d}r}\left[ r \frac{\textrm{d}}{{\rm d}r} K_1(r)\right]  + \left(1+ \dfrac{1}{r^2} \right)K_1(r) = 0 \quad \forall \, r >0,
$$ 
that has the asymptotic expansion:
$$
K_1(r) \sim 
\{
\begin{aligned}
& \sqrt{\dfrac{\pi}{2r}} \exp(-r) & \text{when $r \to \infty$} \\
& \dfrac{1}{r} & \text{ when $r \to 0$}
\end{aligned}
\right. 
$$
We have again that $K_1 \geq 0$ and $K'_1 \leq 0$ on $(0,\infty)$, that the derivatives 
of $K_1$ enjoy the same decay as $K_1$ at infinity and $K'_1(r) \sim -1/r^2$ in $0$.

\medskip

Then, for arbitrary $\lambda >0,$ we define $\phi_{\lambda} : \mathbb R^2 \to \mathbb R$
by:
$$
\phi_{\lambda}(x) = 
\{
\begin{aligned}
& \dfrac{K_0(\sqrt{\lambda}|x|)}{\lambda K_0(\sqrt{\lambda}) -  \dfrac{2\pi}{m} \sqrt{\lambda}K_0'(\sqrt{\lambda})} & \text{ if $|x| > 1$}\\
&  \dfrac{K_0(\sqrt{\lambda})}{\lambda K_0(\sqrt{\lambda}) -  \dfrac{2\pi}{m}\sqrt{\lambda} K_0'(\sqrt{\lambda})}  & \text{ if $|x| \leq 1$}
\end{aligned}
\right.
$$
and $\psi_{\mu} : \mathbb R^2 \to \mathbb R$ by
$$
\psi_{\lambda}(x) = 
\{
\begin{aligned}
& \dfrac{K_1(\sqrt{\lambda}|x|)}{ (2\mathcal{J}^{-1} + \lambda) K_1(\sqrt{\lambda}) -  \pi \mathcal J^{-1}  \sqrt{\lambda}K'_1(\sqrt{\lambda})}  & \text{ if $|x| > 1$}\\
&  \dfrac{K_1(\sqrt{\lambda}) \, \abs{x}}{ (2\mathcal{J}^{-1} + \lambda) K_1(\sqrt{\lambda}) -  2 \pi \mathcal J^{-1}  \sqrt{\lambda}K'_1(\sqrt{\lambda})}  & \text{ if $|x| \leq 1$}
\end{aligned}
\right.
$$
{ We recall that the symbols $m$ and ${\cal J}$ appearing in these formulas stand resepctively for the mass and inertia of the disk.}
The aim of this construction is the following proposition:
\begin{prop}
Let $\lambda >0.$ Given $(F,\tau) \in \mathbb R^2 \times \mathbb R,$ let define:
$$
V_{\lambda}[F,\tau](x) = \phi_{\lambda}(x) F + \psi_{\lambda}(x) \tau \dfrac{x^{\bot} }{|x|}  \quad \forall \, x \in \mathbb R^2.
$$
Then there holds $V_{\lambda}[F,\tau] \in \mathcal{D}(\tilde{\mathcal A})$ with:
$$
(\tilde{\mathcal A}+ \lambda) V_{\lambda}[F,\tau] = (F + \tau x^{\bot}  ) \mathds{1}_{B_0}.
$$
\end{prop}

\begin{proof}
Let $\lambda >0$ and $(F,\tau) \in \mathbb R^2 \times \mathbb R.$ For the proof, we denote
$V =V_{\lambda}[F,\tau]$ for legibility. By construction, $\psi_{\lambda}$ and $\phi_{\lambda}$ are continuous on $\mathbb R^2.$ Furthermore, since $K_0,K_1$ are smooth and decay exponentially at infinity, we have that $\phi_{\lambda},\psi_{\lambda} \in H^2(\mathcal F_0).$ The explicit values for $\phi_{\lambda}$ and $\psi_{\lambda}$ when $r <1$ yield also that, on $B(0,1),$
 we have:
$$
V(x) =  \dfrac{K_0(\sqrt{\lambda})}{\lambda K_0(\sqrt{\lambda}) - \frac{2\pi}{m} \sqrt{\lambda} K'_0(\sqrt{\lambda})} F +  \dfrac{K_1(\sqrt{\lambda})\tau}{(2\mathcal{J}^{-1}+\lambda) K_1(\sqrt{\lambda}) - 2\pi \mathcal J^{-1}\sqrt{\lambda} K'_1(\sqrt{\lambda})}  x^{\bot} \text{ on $B_0$}.
$$
Finally, we obtain that $V \in { L^2[B_0] \cap [H^1(\mathbb R^2)]^2}$ and thus that $V \in \mathcal D(\tilde{\mathcal A}).$

\medskip

We go now to polar coordinates $(r,\theta)$ and exploit the ODE satisfied by $K_0,K_1$ to obtain that 
$$
-\Delta  V + \lambda V = (\tilde{\mathcal A}  +\lambda)V = 0  \quad\text{ in $\mathcal F_0$}.
$$
This is why we introduced Bessel functions.
While, in $B_0,$ we have:
$$
 (\tilde{\mathcal A}  +\lambda) V= \ell + \omega y^{\bot}
$$
with
\begin{align*}
\ell &= \dfrac{1}{m} \int_{\partial B_0} \partial_n v {\rm d}\sigma + \dfrac{ \lambda K_0(\sqrt{\lambda})}{\lambda K_0(\sqrt{\lambda}) - \frac{2\pi}{m} \sqrt{\lambda}K'_0(\sqrt{\lambda})} F  \\
\omega &= \mathcal J^{-1} \left( \int_{\partial B_0} z^{\bot} \partial_nv {\rm d\sigma} +  \dfrac{2 K_1(\sqrt{\lambda})\tau}{(2+\lambda) K_1(\sqrt{\lambda}) - 2\pi \mathcal J^{-1} \sqrt{\lambda}K'_1(\sqrt{\lambda})}   \right) \\
& \quad \quad +   \dfrac{\lambda K_1(\sqrt{\lambda})\tau}{(2\mathcal{J}^{-1}+\lambda) K_1(\sqrt{\lambda}) - 2 \pi \mathcal J^{-1} \sqrt{\lambda}K'_1(\sqrt{\lambda})}. 
\end{align*}

Going again to polar coordinates $(r,\theta),$ we note that $\partial_n = - \partial_r$ and 
that $z^{\bot} = (-\sin(\theta),\cos(\theta)).$ For symmetry reasons, we thus have  that:
\begin{align*}
\int_{\partial B_0} \partial_n v {\rm d}\sigma & = - \dfrac{2 \pi  \sqrt{\lambda } K'_0(\sqrt{\lambda})}{\lambda K_0(\sqrt{\lambda}) - \dfrac{2\pi}{m} \sqrt{\lambda}K'_0(\sqrt{\lambda})} F \\ 
\int_{\partial B_0} z^{\bot} \partial_n v {\rm d}\sigma & = - \dfrac{2\pi  \sqrt{\lambda } K'_1(\sqrt{\lambda})}{(2\mathcal{J}^{-1}+\lambda) K_1(\sqrt{\lambda}) -2  \pi \mathcal J^{-1} \sqrt{\lambda}K'_1(\sqrt{\lambda})} \tau 
\end{align*}
Introducing these identities in the above computations of $\ell$ and $\omega$, we end up with $\ell = F$ and $\omega = \tau.$ 
This concludes the proof.
\end{proof}

We combine now this construction with the operator $\tilde{\mathcal A}_0$ to compute the
resolvant of $\tilde{\mathcal A}.$ Given $\lambda >0$ and $W \in L^2[B_0]$, we have:
$$
W = (\ell_{W} + \omega_{W} y^{\bot}) \mathds{1}_{B_0} + w \mathds{1}_{\mathcal F_0}. 
$$
Consider $V^{(0)}_{\lambda}[W] = (\tilde{\mathcal{A}}_0 + \lambda)^{-1} ( w \mathds{1}_{\mathcal F_0}).$
We have $V^{(0)}_{\lambda} \in \mathcal{D}(\tilde{\mathcal A}_0) \subset \mathcal{D}(\tilde{\mathcal A})$ so that we can compute $\tilde{\mathcal A}V^{(0)}_{\lambda}$:
$$
( \tilde{\mathcal A} +\lambda)(V^{(0)}_{\lambda}[W])  = 
\left[ \dfrac{1}{m} \left( \int_{\partial B_0} \partial_n v^{(0)}_{\lambda}[W] {\rm d}\sigma \right) +  \mathcal{J}^{-1} \left( \int_{\partial B_0} z^{\bot} \cdot \partial_n v^{(0)}_{\lambda}[W] {\rm d}\sigma \right) y^{\bot}\right] \mathds{1}_{B_0} + w\mathds{1}_{\mathcal F_0}.
$$   
Consequently, we correct the value on $B_0$ by setting:
\begin{equation} \label{eq_calculefforts}
F_{\lambda}^{(0)}[W] := \dfrac{1}{m} \left( \int_{\partial B_0} \partial_n v^{(0)}_{\lambda}[W] {\rm d}\sigma \right), \quad 
\tau_{\lambda}^{(0)}[W] :=  \mathcal{J}^{-1} \left( \int_{\partial B_0} z^{\bot} \cdot \partial_n v^{(0)}_{\lambda}[W] {\rm d}\sigma \right),
\end{equation}
and
\begin{equation} \label{eq_Vlambda}
\tilde{V}_{\lambda}[W] := V^{(0)}_{\lambda}[W] +  V_{\lambda}[\ell_{W}-F_{\lambda}^{(0)}[W] ,\omega_{W} - \tau_{\lambda}^{(0)}[W]].
\end{equation}
By linearity, we obtain that $\tilde{V}_{\lambda} \in \mathcal{D}(\tilde{\mathcal A})$ satisfies:
$$
(\tilde{\mathcal A} + \lambda) \tilde{V}_{\lambda}[W] = W. 
$$
and is the unique one by injectivity of $\tilde{\mathcal A}+ \lambda.$

\medskip

With this construction at-hand, we are in position to prove {Proposition \ref{prop:fractional}}.  
\begin{proof}[Proof of Proposition \ref{prop:fractional}]
Fix $q \in (1,2)$ and $\mu < \mu_{crit} =  1/q -1/2.$ Let $\varepsilon >0$ and $W \in L^2[B_0].$ Plugging \eqref{eq_Vlambda} into \eqref{eq_repfract}
we obtain that
\begin{multline*}
(\tilde{A} + \varepsilon )^{-\mu} = \dfrac{\sin(\pi \mu)}{\pi} \int_0^{\infty} \dfrac{1}{(\lambda+ \varepsilon)^{\mu}} (\tilde{\mathcal A}_0 + \varepsilon +\lambda )^{-1}( \mathbf{1}_{\mathcal F_0} W) {\rm d}{\lambda}  \\
+  \dfrac{\sin(\pi \mu)}{\pi} \int_0^{\infty} \dfrac{1}{(\lambda+ \varepsilon)^{\mu} }
{ V_{\varepsilon + \lambda}}[\ell_W - F_{\lambda+\varepsilon}^{0}[W],\omega_{W} + \tau_{\lambda+ \varepsilon}^{0}[W]]{\rm d}\lambda.
\end{multline*}
Thus, we have the expected representation formula with:
$$
R_{\mu,\varepsilon}W = \dfrac{\sin(\pi \mu)}{\pi} \int_0^{\infty} \dfrac{1}{(\lambda+ \varepsilon)^{\mu} 
}{ V_{\varepsilon + \lambda}}[\ell_W - F_{\lambda+\varepsilon}^{(0)}[W],\omega_{W} + \tau_{\lambda+ \varepsilon}^{(0)}[W]]{\rm d}\lambda.
$$
To complete the proof, it remains to obtain \eqref{eq_controlfrac}. For this, we first bound
by introducing the explicit value of ${V_{\varepsilon+\lambda}}$:
\begin{align*}
\|R_{\mu,\varepsilon} W\|_{\mathcal L^2}
\leq C \int_0^{\infty} \dfrac{1}{(\lambda+ \varepsilon)^{\mu}} \|\phi_{\lambda+\varepsilon}\|_{L^2}(|\ell_{W}| + |F^{(0)}_{\lambda+\varepsilon}[W]|){\rm d}\lambda 
\\
 +C\int_0^{\infty} \dfrac{1}{(\lambda+ \varepsilon)^{\mu}}  \|\psi_{\lambda+\varepsilon}\|_{L^2}(|\omega_{W}| + |\tau^{(0)}_{\lambda+\varepsilon}[W]|) {\rm d}\lambda.
\end{align*}
We note here that the constant $C$ appearing in the right-hand side depends only on the physical parameters of the system. We denote by $C$ such constants below. They can depend on the physical parameters or on the data $q,\mu.$ They can also vary between lines.

We proceed by estimating the two integrals independently. For the first one, let denote:
$$
K(s) = \dfrac{1}{s^{\mu}} \|\phi_{s}\|_{L^2(\mathbb R^d)}(|\ell_{W}| + |F^{(0)}_{s}[W]|)
$$
By looking at the explicit value of $\phi_{s}$, we have:
$$
 \|\phi_{s}\|_{L^2(\mathbb R^d)}
 \leq \dfrac{C}{s K_0(\sqrt{s}) - \sqrt{s} K'_0(\sqrt{s})} \left( K_0(\sqrt{s}) + \dfrac{1}{\sqrt{s}}\left(\int_{\sqrt{s}}^{\infty} |K_0(\alpha)|^2\alpha{\rm d}\alpha \right)^{\frac 12}\right).
$$
and, with $q'$ the conjugate exponant of $q:$
\begin{equation} \label{eq_controlF}
|F_{s}^{(0)}[W]| \leq C \dfrac{\|w\|_{L^q(\mathcal F_0)}}{K_0(\sqrt{s})s^{\frac 1{q'}} } \left( \int_{\sqrt{s}}^{\infty} |K_{0}(\alpha)|^{q'}\alpha{\rm d}\alpha \right)^{\frac 1{q'}} .
\end{equation}
We postpone the proof of this latter inequality to the end of the appendix. 

\medskip

When $s \in (0,1)$ the asymptotics of $K_0$ and $K_0'$
ensure that $K_0 \in L^p ((0,\infty))$ for all $p \geq 1$ and that
\begin{align*}
K(s) & \leq \dfrac{C}{s^{\mu}}\dfrac{\|W\|_{L^q}}{\sqrt{s}|K_0'(\sqrt{s})|} \left(1+ \left(s^{\frac 1{q'}}K_0(\sqrt{s})\right)^{-1}\right) \left(K_0(\sqrt{s}) + \dfrac{1}{\sqrt{s}}\right)\\
& \leq \dfrac{C\|W\|_{L^q}}{s^{\mu-\mu_{crit}+1}|\ln(s)|},
\end{align*}
where $1/s^{\mu-\mu_{crit}+1}|\ln(s)|\in  L^1((0,1))$ since $\mu - \mu_{crit} < 0.$ While, when $s \in (1,\infty)$, the same asymptotics guarantee that
{(remember that $q'>2$ to bound $\alpha^{1-q'/2} \leq s^{1/2-q'/4}$ for $\sqrt{s}<\alpha$ )}:
\begin{multline*}
K(s) \leq \dfrac{C\|W\|_{L^q} }{s^{\mu}} 
\dfrac{\left(1+ \exp(\sqrt{s}) s^{ - \frac 1{2q'}} (\int_{\sqrt{s}}^{\infty} \exp (-q'\alpha) {\rm d}\alpha)^{\frac 1{q'}}\right)}{s^{\frac 34}\exp(-\sqrt{s})} \dots \\
\dots \left( \dfrac{\exp(-\sqrt{s})}{s^{\frac 14}} + \dfrac{1}{\sqrt{s}} \left( \int_{\sqrt{s}}^{\infty}\exp(-2\alpha){\rm d}\alpha \right)^{\frac 12}\right)
\end{multline*}
and finally $K(s) \leq C\|W\|_{L^q(\mathbb R^2)}  s^{-\mu - \frac 1{2q'} - 1} \in L^1(1,\infty).$
Hence, we have a uniform bound $C$ independent of $\varepsilon \in (0,1)$ such that:
$$
 \int_0^{\infty} \dfrac{1}{(\lambda+\varepsilon)^{\mu}} \|\phi_{\lambda+\varepsilon}\|_{L^2(\mathbb R^d)}(|\ell_{W}| + |F^{(0)}_{\lambda+\varepsilon}[W]|){\rm d}\lambda  \leq \int_0^{\infty} K(s){\rm d}s  \leq   C \|W\|_{L^q} .
$$ 

\medskip
For the second integral we denote similarly:
$$
\tilde{K}(s)  =\dfrac{1}{s^{\mu}}  \|\psi_{s}\|_{L^2(\mathbb R^2)}(|\omega_{W}| + |\tau^{(0)}_{s}[W]| )
$$
With the explicit form of $\psi_s$ we have:
$$
\|\psi_{s}\|_{L^2(\mathbb R^2)} \leq \dfrac{C}{(1+s)K_1(\sqrt{s}) - \sqrt{s} K_1'(\sqrt{s})} \left( |K_1(\sqrt{s})| + \dfrac{1}{\sqrt{s}}\left( \int_{\sqrt{s}}^{\infty} |K_1(\alpha)|^2 \alpha {\rm d}\alpha \right)^{\frac 12}\right).
$$
and 
\begin{equation} \label{eq_controltau}
|\tau_{s}^{(0)}[W]|  \leq C \dfrac{\|w\|_{L^q(\mathcal F_0)}}{K_1(\sqrt{s})s^{\frac 1{q'}} } \left( \int_{\sqrt{s}}^{\infty} |K_{1}(\alpha)|^{q'}\alpha{\rm d}\alpha\right)^{\frac 1{q'}} 
\end{equation}

When $s \in (1,\infty)$, $K_0$ and $K_1$ admit a similar exponential bound, so we obtain with similar arguments as previously that $\tilde{K}$ is dominated by an $L^1$-function multiplied by $\|W\|_{L^q(\mathbb R^2)}$.
When $s \in (0,1)$ we proceed more carefully but similarly again. We have $|K_1(\alpha)| \leq 1/\alpha$  when $\alpha <1.$ Hence, we compute that:
$$
\int_{\sqrt{s}}^{\infty} |K_1(\alpha)|^{2}\alpha{\rm d}\alpha \leq {C(1+| \ln(s)|)}, \qquad \int_{\sqrt{s}}^{\infty} |K_1(\alpha)|^{q'}\alpha{\rm d}\alpha \leq \dfrac{1}{s^{\frac {q'}{2} - 1}}.
$$ 
Consequently:
$$
|\tilde{K}(s)| \leq { C(1+ |\ln(s)|)^{\frac 12} \|W\|_{L^q(\mathbb R^2)}.}  
$$
We conclude like previously.
\end{proof}  

To end up this section, we provide a proof of identities \eqref{eq_controlF}-\eqref{eq_controltau} . This is the content of the following proposition:

\begin{prop}
Let $\lambda >0$ and $q \in (1,2)$. 
There exists a constant $C$ depending only on the physical parameters and $q$ such that, given $W \in L^2[B_0] \cap [L^q(\mathbb R^2]^2$ we have:
\begin{align*}
|F_{\lambda}^{(0)}[W]| & \leq C \dfrac{\|w\|_{L^q(\mathcal F_0)}}{K_0(\sqrt{\lambda})\lambda^{\frac 1{q'}} } \left( \int_{\sqrt{\lambda}}^{\infty} |K_{0}(s)|^{q'}s{\rm d}s\right)^{\frac 1{q'}} \\
|\tau_{\lambda}^{(0)}[W]| & \leq C \dfrac{\|w\|_{L^q(\mathcal F_0)}}{K_1(\sqrt{\lambda})\lambda^{\frac 1{q'}} } \left( \int_{\sqrt{\lambda}}^{\infty} |K_{1}(s)|^{q'}s{\rm d}s\right)^{\frac 1{q'}} 
\end{align*}
where $F_{\lambda}^{(0)}[W]$ and $\tau_{\lambda}^{(0)}[W]$ are defined in \eqref{eq_calculefforts}
and $q'$ is the conjugate exponent of $q.$
\end{prop}
\begin{proof}
We provide a proof of the second inequality. The first one is obtained with a similar construction based on $K_0$. 

\medskip

Let $\omega \in \mathbb R$ and 
$$
u(x) = \dfrac{K_{1}(\sqrt{\lambda}|x|)}{K_1(\sqrt{\lambda})} \dfrac{\omega x^{\bot}}{|x|} \quad \forall \, x \in \mathcal F_0.
$$
By construction, we have:
$$
-\Delta u + \lambda u = 0 \text{ on $\mathcal F_0$}, \quad u(x) = \omega x^{\bot} \text{ on $\partial B_0$}.
$$
Introducing the latter identity into the definition of $\tau_{\lambda}^{(0)}[W]$, we derive: 
$$
\tau_{\lambda}^{(0)}[W]\omega = \int_{\partial B_0} \partial_n v^{(0)}_{\lambda}[W] u {\rm d}\sigma,
$$
so that we can integrate by parts. Recalling that $v_{\lambda}^{(0)}[W] = V_{\lambda}^{(0)}[W]\mathds{1}_{\mathcal F_0}$ satisfies a specific PDE and vanishes on $\partial B_0$, then using the PDE satisfied by $u$, we deduce successively that:
\begin{align*}
\tau_{\lambda}^{(0)}[W]\omega & = \int_{\mathcal F_0} \Delta v^{(0)}_{\lambda} \cdot u + \int_{\mathcal F_0} \nabla v_{\lambda}^{(0)} :\nabla u \\
& = - \int_{\mathcal F_0} w \cdot u + \int_{\mathcal F_0} \lambda v^{(0)}_{\lambda} \cdot u - \int_{\mathcal F_0}  v^{(0)}_{\lambda} \cdot \Delta u   \\
&= - \int_{\mathcal F_0} w \cdot u .
\end{align*}
Via a standard H\"older inequality and homogeneity arguments we thus infer that:
$$
|\tau_{\lambda}^{(0)}[W] \omega| \leq \dfrac{\|w\|_{L^{q}(\mathcal F_0)} |\omega|}{K_1(\sqrt{\lambda})\lambda^{\frac 1{q'}}} \left( \int_{\sqrt{\lambda}}^{\infty} |K_{1}(s)|^{q'}s{\rm d}s\right)^{\frac 1{q'}} .
$$
Since $\omega$ is arbitrary, this concludes the proof.
\end{proof}

{\bf Acknowledgements.} The second author acknowledges support of the Institut Universitaire de France and project "SingFlows" ANR-grant number: ANR-18-CE40-0027.  The second author would also like to thank Luc Hillairet for pointing reference [10] on special functions and for enlightening discussions about reference \cite{KozonoOgawa-decay}.

\end{document}